\documentclass{amsart}
\usepackage{amsmath}
\usepackage{amsfonts}
\usepackage{amssymb}
\usepackage{amsthm}
\usepackage{verbatim}
\usepackage{epsfig,psfrag}
\usepackage{graphicx}
\usepackage{color}
\usepackage[dvips]{xcolor}

\definecolor{Blue}{rgb}{0.3,0.3,0.9}
\definecolor{orange}{rgb}{1,0.5,0}

\newcommand{\al}{\alpha}
\newcommand{\be}{\beta}

\newcommand{\de}{\delta}
\newcommand{\De}{\Delta}
\newcommand{\vth}{\vartheta}

\newcommand{\ep}{\epsilon}
\newcommand{\vep}{\varepsilon}
\newcommand{\la}{\lambda}

\newcommand{\z}{\zeta}
\newcommand{\om}{\omega}
\newcommand{\Om}{\Omega}

\renewcommand{\for}{\quad \text{for}\ }
\newcommand{\on}{\ \text{on}\ }

\newcommand{\inn}{\ \text{in}\ }

\newcommand{\ifff}{\quad \text{if}\ }
\newcommand{\andy}{\quad \text{and}\quad }
\newcommand{\with}{\quad \text{with}\ }

\newcommand{\where}{\quad \text{where}\ }

\newcommand{\wt}{\widetilde}
\newcommand{\wh}{\widehat}

\newcommand{\fy}{\phi}
\newcommand{\vfy}{\varphi}

\renewcommand{\H}{\mathcal{H}}
\newcommand{\T}{\mathcal{T}}

\newcommand{\K}{\mathcal{K}}

\newcommand{\R}{\mathbb{R}}
\newcommand{\A}{{A}}
\newcommand{\bA}{{\bf A}}

\newcommand{\C}{\mathcal{C}}

\newcommand{\tu}{\wt u}
\newcommand{\tU}{\wt U}
\newcommand{\bfu}{{\bf u}}
\newcommand{\bfv}{{\bf v}}
\newcommand{\bfU}{{\bf U}}

\newcommand{\wtDe}{{\wt \Delta}}

\newcommand{\wtE}{{\wt E}}

\newcommand{\tpj}{{\tilde \phi}_j^h}
\newcommand{\tlj}{{\tilde \lambda}_j^h}
\newcommand{\tde}{\wt \delta}

\newcommand{\tribar}{\vert\thickspace\!\!\vert\thickspace\!\!\vert}
\newcommand{\lla}{\langle}
\newcommand{\rra}{\rangle}
\newcommand{\Id}{{I}}

\newcommand{\Hdot}{\dot H}
\def\L2o{{L_2(\Om)}}
\def\K{\tau}

\def\T{{\mathcal{T}}}

\def\ww{Q_h \psi}
\def\w{\psi}

\def\I{J}

\def\Ah{\wt{ A}_h}

\def\vv{\psi}

\def\tveph{\wt{\varepsilon}_h}
\def\tah{\wt {a}_h}
\def\Pio0{{\Pi_{\z_0}}}
\def\wPio0{\widehat{\Pi}_{\z_0}}
\def\wtPio0{\widetilde{\Pi}_{\z_0}}
\def\Pioi{\Pi_{x_i}}

\def \qc{{\beta}}
\def \qbar{\widetilde{\beta}}

\def \all{{\alpha}}
\def \alb{{\widetilde \alpha}}

\def \Qal{\widetilde{Q}_h}
\def \detil{\widetilde{\delta}}
\def \dehat{\widehat{\delta}}

\numberwithin{equation}{section}
\newtheorem{theorem}{Theorem}[section]
\newtheorem{lemma}{Lemma}[section]
\newtheorem{corollary}{Corollary}[section]
\newtheorem{remark}[theorem]{Remark}
\newtheorem{proposition}{Proposition}[section]

\begin{document}
\title[Error estimates for the  FVEM]
{Some error estimates for the  finite volume element method for a
parabolic problem}

\author[]{P. Chatzipantelidis}

\address{Department of Mathematics,
University of Crete, GR--71409 Heraklion, Greece}
\email{chatzipa@math.uoc.gr}

\author[]{R.D. Lazarov}

\address{Department of Mathematics,
Texas A\&M University, College Station, TX--77843,  USA and
Institute of Mathematics and Informatics, Bulgarian Acad. Sciences, 
Acad. G.Bonchev str., bl.8, 1113 Sofia, Bulgaria}
\email{lazarov@math.tamu.edu}

\author[]{V. Thom\'ee}

\address{Mathematical Sciences, Chalmers University of Technology and the
University of Gothenburg, SE-412 96 G\"oteborg, Sweden,  and Institute of Applied and
Computational Mathematics, FORTH, Heraklion GR--71110, Greece}
\email{thomee@chalmers.se}


\date{started October, 2010; today is \today}


\begin{abstract}
We study spatially semidiscrete and fully discrete finite volume
element methods for the homogeneous heat equation with
homogeneous Dirichlet boundary conditions and derive error estimates
for smooth and nonsmooth initial data. We show that the results of
our earlier work \cite{clt11} for the lumped mass method carry over to
the present situation. In particular, in order for error estimates for
initial data only in $L_2$ to be of optimal second order for positive
time, a special condition is required, which is satisfied for
symmetric triangulations. Without any such condition, only first order
convergence can be shown, which is  illustrated by a counterexample.
Improvements  hold for triangulations that are almost symmetric and
piecewise almost symmetric.
\end{abstract}
\maketitle

\setcounter{equation}{0}


\section{Introduction}\label{sec:intro}

We consider the model initial--boundary value problem
\begin{equation}\label{eq1}
 u_t-\De u= 0, \inn\Om,\quad u=0,\on
\partial\Om, \for t\ge0,
\with 
u(0)=v,\inn\Om,
\end{equation}
where $\Om$  is a bounded convex polygonal domain in $\mathbb
R^2$.  We restrict ourselves to
 the homogeneous heat equation, thus without a forcing term,
 so that the initial values $v$ are the only data of the problem.
This problem has a unique solution $u(t)$, under  appropriate
assumptions on $v$, and this solution is smooth for $t>0$, even if
$v$ is not.

To express the smoothness properties of the solution of \eqref{eq1}, let, for $q\ge0$,
 $\dot H^q\subset L_2(\Om)$ be the Hilbert space defined by the norm
\begin{equation}\label{norms}
|w|_q=\Big(\sum_{j=1}^\infty\la_j^q(w,\fy_j)^2\Big)^{1/2}, \where
(w,\varphi)=\int_\Om w \varphi \,dx,
\end{equation}
and where $\{\la_j\}_{j=1}^\infty$, $\{\fy_j\}_{j=1}^\infty$ are the
eigenvalues, in increasing order, and orthonormal   
eigenfunctions of $-\De$ in $\Om$,
with homogeneous 
Dirichlet boundary conditions on $\partial\Om$. 
{Thus $|w|_0=\|w\|=(w,w)^{1/2}$
is the norm in $L_2=L_2(\Om)$, $|w|_1 =\|\nabla w\|$ the norm in
$H_0^1=H_0^1(\Om)$ and $|w|_2=\|\De w\|$ is equivalent to the norm
in $H^2(\Om)$ when $w=0$ on $\partial\Om$. 
}
Eigenfunction expansion
and Parseval's relation shows for 
 the solution $u(t)=E(t)v$ of
\eqref{eq1}  the stability and smoothing estimate
\begin{equation}\label{1.smooth}
|E(t)v|_p\le Ct^{-(p-q)/2}|v|_q, \for 0\le q\le p, \andy t>0.
\end{equation}
In fact, since the smallest eigenvalue is positive, a factor of
$e^{-ct}$, with $c>0$, may be included in the right hand side, and
this holds for all our stability, smoothing and error estimates
throughout our paper. Since our interest here is in small time we
shall not keep track of this decay for large time below.
{
We shall also use the norm $\| w \|_{\C^k}=\sum_{|\gamma|\le
k}\sup_{x\in\Om} |D_x^{\gamma}w(x)|$ in $\C^k=\C^k(\overline \Om)$, with
$\C=\C^0$, the space of
continuous functions on $\overline \Om$.  Here for 
$\gamma =(\gamma_1, \gamma_2)$,
$D_x^\gamma=(\partial/\partial x_1)^{\gamma_1}(\partial/\partial x_2)^{\gamma_2}$
and $|\gamma|=\gamma_1+\gamma_2$.
}

We first recall some facts about the spatially semidiscrete standard
Galerkin finite element method for \eqref{eq1} in the space of piecewise linear functions
\[
S_h =\{\chi\in \C: \ \chi|_\K\ \text{ linear},\quad \forall\, \K
\in \T_h; \  \chi|_{\partial \Om}=0\},
\]
where ${\{\T_h\}}$ is a family of regular triangulations
$\T_h=\{\tau\}$ of $\Omega$, with $h$ denoting the maximum diameter of
the triangles $\tau\in \T_h$. This method defines an approximation
$u_h(t)\in S_h$ of $u(t)$, for $t\ge0$, from
\begin{equation}\label{fem}
(u_{h,t},\chi) +(\nabla u_h,\nabla\chi)=0,\quad\forall\chi\in S_h,
\for t\ge0,\with u_h(0)=v_h,
\end{equation}
where $v_h \in S_h$ is an approximation of  $v$. It is well--known
that we have the smooth data error estimate, valid uniformly down to $t=0$,
 see e.g. \cite{Thomee06}, 
\begin{equation}\label{1.sm}
\|u_h(t)-u(t)\|\le Ch^2|v|_2,\  \ifff\  \|v_h-v\|\le
Ch^2|v|_2, \for t\ge0.
\end{equation}
We also have a nonsmooth data error estimate, for $v$  only assumed
to be in $L_2$,  which is of
optimal order $O(h^2)$ for $t$ bounded away from zero, but
deteriorates as $t\to0$,
\begin{equation}\label{1.nsm}
\|u_h(t)-u(t)\|\le Ch^2t^{-1}\|v\|,\  \ifff \ v_h=P_hv, \for t>0,
\end{equation}
where $P_h$ denotes the orthogonal $L_2-$projection onto $S_h$. Note
that the choice of
discrete initial data is not as general in this case as in
\eqref{1.sm}. We emphasize that the triangulations $\T_h$
are assumed to be independent of $t$, and thus the use of finer
$\T_h$ for $t$ small is not considered here.

We note  that a possible choice in \eqref{1.sm} is
$v_h=P_hv$, and  hence, by interpolation, we have the
intermediate result between \eqref{1.sm} and \eqref{1.nsm},
\begin{equation}\label{1.hsm}
\|u_h(t)-u(t)\|\le Ch^2t^{-1/2}|v|_1,\  \ifff v_h=P_hv, \for t>0.
\end{equation}

Recently, in \cite{clt11}, we showed  results similar to
\eqref{1.sm}--\eqref{1.hsm} for the lumped mass finite element method,
which may be defined by replacing the $L_2-$inner product in the first term
in \eqref{fem} by the quadrature approximation $(u_{h,t},\chi)_h$, where,
with $I_h:\C\to S_h$ 
{being} the interpolant defined by $I_hv(z)=v(z)$ for any vertex
$z$ of $\T_h$,
\[
(\chi,\psi)_h=\int_\Om I_h(\chi\,\psi)\,dx,\quad\forall \chi,\psi\in S_h.
\]
Improving earlier results, we  demonstrated
 that \eqref{1.sm} remains valid for the lumped mass method, but that
\eqref{1.nsm} requires restrictive conditions on $\{\T_h\}$,
caused by the use of quadrature in  \eqref{fem},  and
 satisfied, in
particular, for symmetric triangulations.
We  remark that the choice of discrete initial
data in the analogue of \eqref{1.hsm} was incorrectly
stated in \cite{clt11}, 
see Section \ref{sec:smooth} below.
\medskip

In the present paper our purpose is to carry over  the analysis in
\cite{clt11} to the finite volume element method for problem
\eqref{eq1}. This method is based  on a local conservation property
associated with the differential equation. Namely, integrating
\eqref{eq1} over any region $V\subset\Omega$ and using Green's
formula, we obtain
\begin{equation}\label{fv-general}
\begin{split}
\int_{V}u_t\,dx-\int_{\partial V}\nabla u\cdot
n\,d\sigma&= 0, 
 \for t \ge 0,
\end{split}
\end{equation}
where $n$ denotes the unit exterior normal vector to $\partial V$.
The  semidiscrete finite volume element approximation $\tu_h(t) \in S_h$, will satisfy 
\eqref{fv-general} for $V$ in a
finite collection of subregions of $\Om$ called control volumes,
the number of which will be equal to the dimension of the finite
element space $S_h$. These control volumes are constructed in the
following way. Let $z_\tau$ be the  barycenter of $\K\in \T_h$. We
connect $z_\K$ by line segments to the  midpoints of the edges of
$\K$, thus partitioning $\K$ into three quadrilaterals $\K_z$, $z\in
Z_h(\K)$, where $Z_h(\K)$ are the vertices of $\K$. Then with each
vertex $z\in Z_h=\cup_{\K\in \T_h}Z_h(\K)$ we associate a control volume $V_z$, which consists of
the union of the subregions $\K_z$, sharing the vertex $z$ (see
Figure \ref{fig-fv1}, left). We denote  the set of interior vertices of
$Z_h$ by $Z_h^0$. The semidiscrete finite volume element method 
for \eqref{eq1} is then to find $\tu_h(t)\in S_h$ such that
\begin{equation}\label{eq4}
\int_{V_z} \tu_{h,t}\,dx-\int_{\partial V_z}\nabla \tu_h\cdot
n\,d\sigma=0, 
\quad \forall z\in Z_h^0,\for t \ge 0,\with
\tu_h(0)=v_h,
\end{equation}
where $v_h\in S_h$ is an approximation of $v$.
\par
\begin{figure}[t]
\par
\centerline{ \psfig{figure=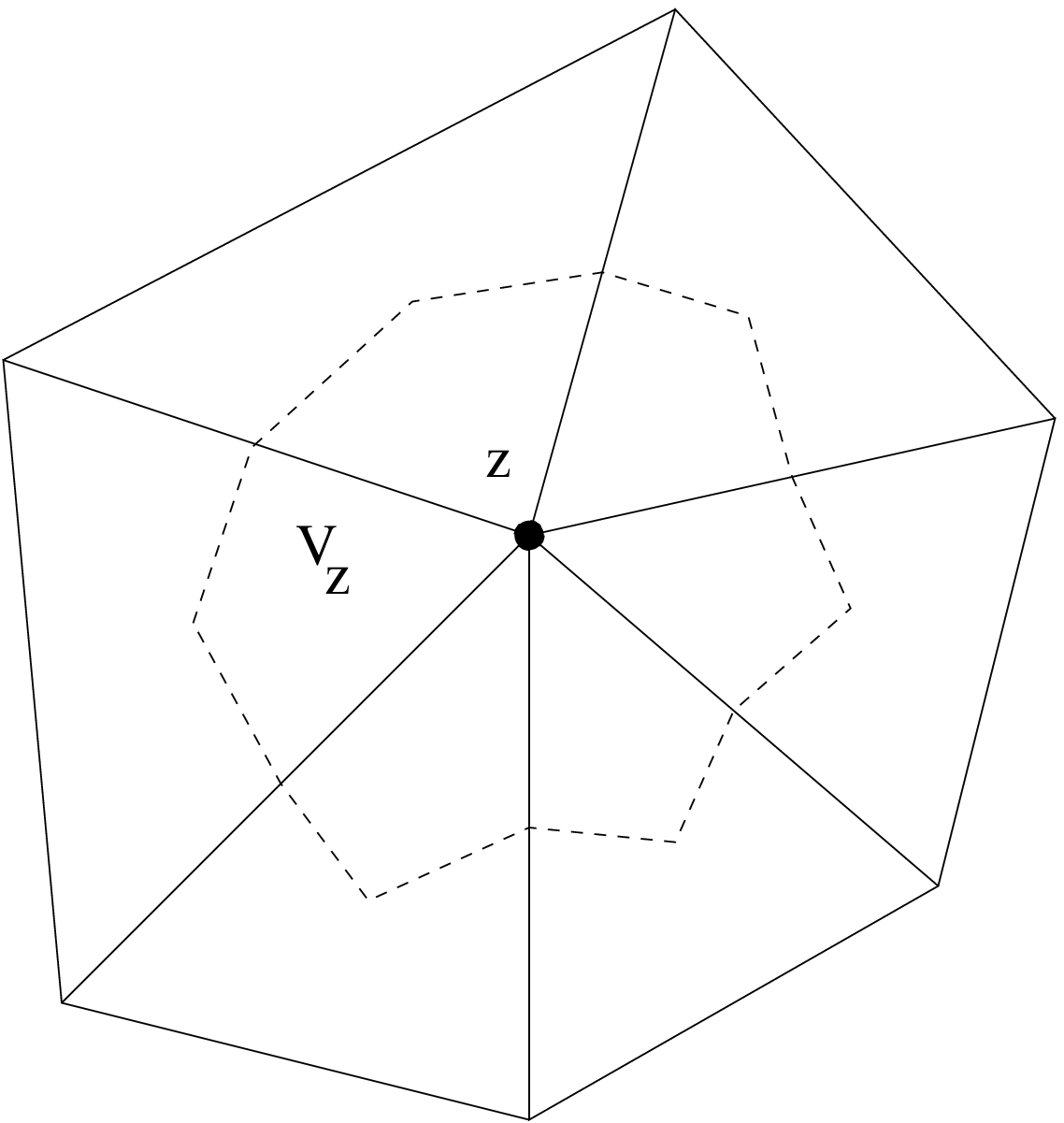,height=1.6in} \hspace{1cm}
\psfig{figure=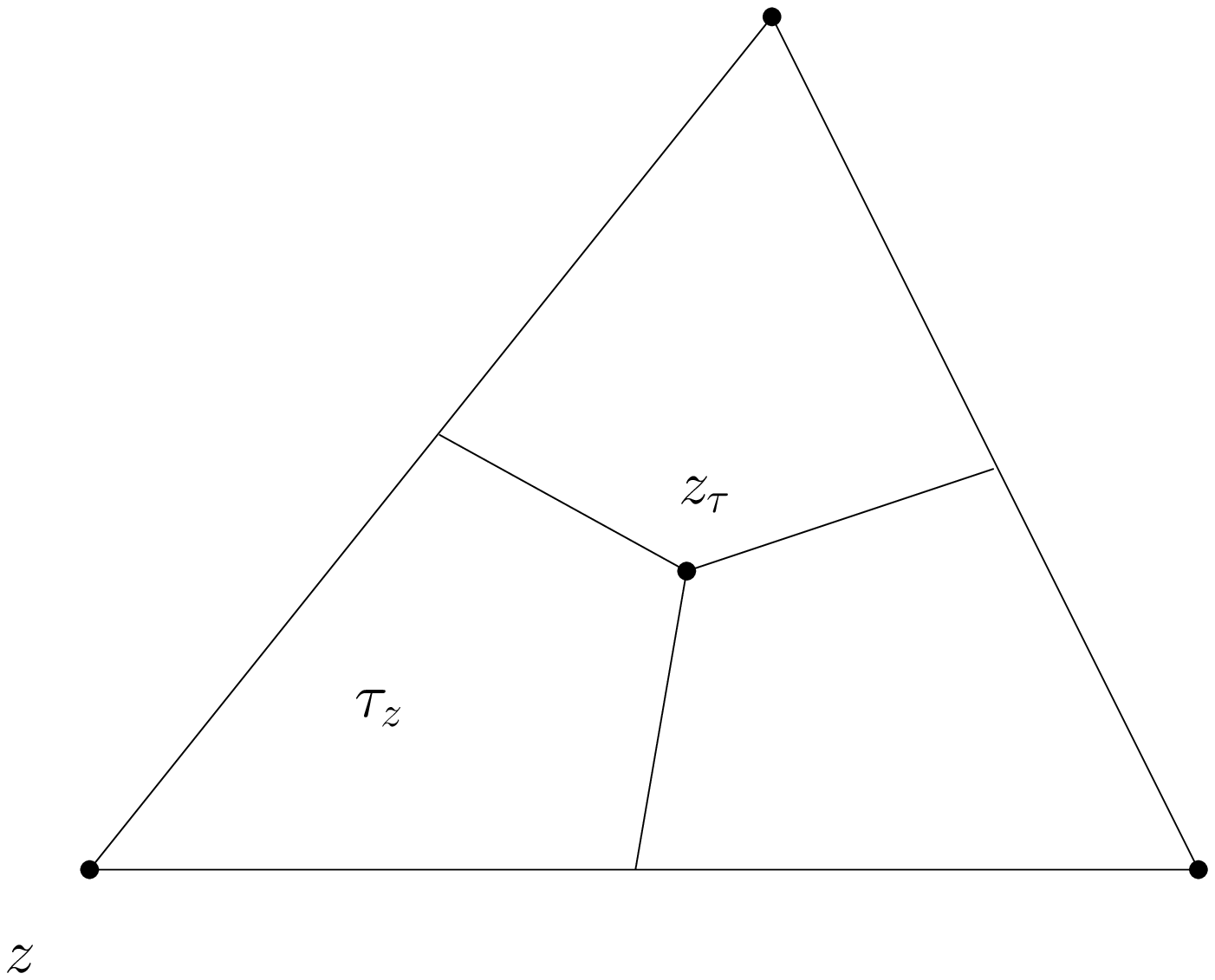,height=1.5in} }
\par
\caption { {\it Left:} A union of triangles that have a common
vertex $z$; the dotted line shows the boundary of the corresponding
control volume $V_z$. {\it Right:} A triangle $\K$ partitioned into
the three subregions $\K_z$.} \label{fig-fv1}
\end{figure}
\par

This problem may also be expressed in a weak form. For this purpose
we introduce the finite--dimensional space of piecewise constant functions
$$
Y_h= \{\eta\in L_2:\ \eta|_{V_z} \ \text{= constant},\ \forall z\in
Z_h^0;\ \eta|_{V_z}=0,\ \forall z\in Z_h\setminus Z_h^0\}.
$$
We now multiply \eqref{eq4} by $\eta(z)$ for an arbitrary $\eta \in
Y_h$,  and  sum over $z \in Z_h^0$ to obtain
the Petrov--Galerkin formulation
\begin{equation}\label{fv}
{(\tu_{h,t},\eta)}+ a_h(\tu_h,\eta)=0, 
\quad\forall \eta\in
Y_h, \for t \ge 0,\with \tu_h(0)=v_h,
\end{equation}
where  the bilinear form
$a_h(\cdot,\cdot):S_h\times Y_h\to \mathbb R$ is defined by
\begin{equation}\label{a-h0}
a_h(\chi,\eta)=-\sum_{z\in Z_h^0} \eta(z)\int_{\partial V_z}\nabla
\chi\cdot n\,d\sigma,\quad \forall \chi\in S_h,\  \eta\in Y_h.
\end{equation}
Obviously,  we can define $a_h(\cdot,\cdot)$ also for
$\chi$ replaced by $w\in H^2$, and using
Green's formula we then easily see that 
\begin{equation*}
a_h(w,\eta)=-(\De w,\eta),\quad\forall w\in H^2,\ \eta\in Y_h.
\end{equation*}

We shall now rewrite the  Petrov--Galerkin method \eqref{fv} as a
Galerkin method in $S_h$. For this purpose, we introduce  the interpolation
operator
 $\I_h: \C \mapsto Y_h$ by
$$
 \I_h u =
\sum_{z \in Z_h^0} u(z)\Psi_z,
$$
where  $\Psi_z$  is the characteristic function of the control volume
$V_z$. It is known 
that $\I_h$ is selfadjoint and positive definite, see \cite{cl00},
and hence  the following defines an inner product 
$\lla\cdot,\cdot\rra$ on $S_h$,
\begin{equation}
\label{1.fvip}
\lla\chi,\psi\rra=(\chi,\I_h\psi),\quad\forall \chi,\psi\in S_h.
\end{equation}
Also, the corresponding discrete norm is equivalent to the $L_2-$norm, 
uniformly in $h$, i.e., with $C \ge c>0$,
 \begin{equation*}
 c{\|\chi\|}\le {\tribar\chi\tribar}\le
 C{\|\chi\|},\quad\forall\chi\in S_h, \where
{\tribar\chi\tribar}\equiv{\lla\chi,\chi\rra}^{1/2},
 \end{equation*}
see \cite{cl00}. Further, in  \cite{Bank-rose}, it is shown that
 \begin{equation*}
 a_h(\chi, \I_h\psi)=(\nabla\chi,\nabla\psi),\quad\forall \chi,\psi\in S_h,
 \end{equation*}
and therefore, $a_h(\cdot,\cdot) $ is symmetric and $a_h(\chi, \I_h\chi)= \|\nabla \chi\|^2$, 
for $\chi\in S_h$.

With this notation, \eqref{fv} may equivalently be written in
Galerkin form as
 \begin{equation}\begin{split}\label{vvv}
\lla\tu_{h,t}, \chi\rra+ (\nabla\tu_h, \nabla\chi)&=0,\quad\forall
\chi\in S_h,\for t\ge0,\with \tu_h(0)=v_h.\end{split}
\end{equation}

Our aim  is thus to show analogues of \eqref{1.sm}--\eqref{1.hsm} 
for the solution of \eqref{vvv},  with the
appropriate choices of $v_h$, i.e.,
\begin{equation}\label{lm-estimates-1}
\|\tu_h(t)-u(t)\|\le Ch^2t^{-1+q/2}|v|_q, \for t>0,\quad q=0,1,2.
\end{equation}
This will be done below for $q=2$, and in the case $q=1$ under the additional
assumption that $\{\T_h\}$ is quasiuniform. 
However, for $q=0$, as in \cite{clt11}, we are only
able to
show \eqref{lm-estimates-1} under
 an additional hypothesis, expressed in terms of the
quadrature error operator $Q_h:S_h\to S_h$, defined by
\begin{equation}\label{q_h-def}
(\nabla Q_h\psi,\nabla \chi)=\vep_h(\psi,\chi),\quad\forall
\chi,\psi\in S_h,
\end{equation}
where $\vep_h(\cdot,\cdot)$ is the quadrature error defined here by
\begin{equation}\label{veph-def}
\vep_h(f,\chi)=( f,\I_h\chi)-{(f,\chi)},
\quad \forall f\in L_2,\ \chi \in S_h,
\end{equation}
and requiring
\begin{equation}\label{higher_error}
\| Q_h \psi \| \le C h^2 \|\psi\|,  \quad  \forall \psi \in S_h.
\end{equation}

We will show that this assumption is satisfied for {\it symmetric}
triangulations $\T_h$. Symmetry of $\T_h$, however,
 is a severe restriction which can
only hold for special shapes of $\Om$. For this reason
 we will also consider less restrictive families $\{\T_h\}$. 
We will demonstrate that 
\eqref{higher_error} holds for {\it almost symmetric} families
(discussed in Section \ref{sec:specialmeshes}), with the  addition of a
logarithmic factor; we also show that
this logarithmic factor is not needed in one space dimension. 
Further, for {\it piecewise almost symmetric} families of
triangulations, see Section \ref{sec:specialmeshes},
 the inequality \eqref{higher_error} holds with an $O(h^{3/2})$ bound.

We  then give two examples of  nonsymmetric
triangulations such that \eqref{lm-estimates-1} does not hold for
$q=0$.  
In the first example we construct $\{\T_h\}$ such that the convergence
factor is
at most of order $O(h)$ for $t>0$, and in the second example, with
nonsymmetry only along
a line, of order $O(h^{3/2})$. 
Without any additional condition on $\T_h$ we are only able to show the
nonoptimal order error estimate
\[
\| \tu_h(t)-u(t)\|\le Cht^{-1/2}\|v\|,\quad \ifff v_h=P_hv, \for t>0.
\]

We remark that in \cite{sel}, in the more general case of a parabolic
integro--differential equation, the nonsmooth data error estimate
\eqref{lm-estimates-1}, for $q=0$,
with an extra factor $|\log h|$,
was stated, for any quasiuniform family $\{\T_h\}$. Unfortunately,
this result is in contradiction to our above counterexamples, and its
proof incorrect.

We also discuss optimal order $O(h)$  error estimates for  the
gradient of $\tu_h-u$, under various assumptions on 
the smoothness of $v$ and choices of $v_h$. Further, in a separate section, we 
consider briefly the extension of our results for the spatially
semidiscrete problem to the fully discrete backward Euler and
Crank--Nicolson  finite volume  methods.

As for the lumped mass method in \cite{clt11}, our 
analysis yields improvements of earlier
results, in \cite{chatzipa-l-thomee04}, where it was shown that, for
smooth initial data and $v_h=R_hv$,
\[
\| \tu_h(t)-u(t)\|\le Ch^2|v|_3,  \for t>0,
\]
and
\[
\| \nabla(\tu_h(t)-u(t))\|\le Ch\epsilon^{-1}|v|_{2+\epsilon},
  \for t>0, \ \epsilon>0\ \text{small}.
\]
As in the case of the lumped mass  method in \cite{clt11}, these improvements 
are  made possible by combining, the error estimates \eqref{1.sm}--\eqref{1.hsm} 
for the standard Galerkin finite element method with bounds for the difference
$\de=\tu_h-u_h$, which, by \eqref{vvv} and \eqref{fem}, satisfies
\begin{equation}\label{1.del}
\lla\de_t,\chi\rra+(\nabla\de,\nabla\chi)=-\vep_h(u_{h,t},\chi),\quad \forall
\chi\in S_h, \for t\ge0.
\end{equation}

In the final section we sketch the extension  of the theory developed 
above  to more general parabolic equations, considering
the initial--boundary value problem
\begin{equation}\label{1.eq1-general}
 u_t+\A u= 0, \inn\Om,\quad
u=0,\on
\partial\Om, \for t\ge0, \with
u(0)=v,\inn\Om,
\end{equation}
where $\A u=-\nabla\cdot(\all\nabla u)+ \qc u$, with $\al$ a
smooth symmetric,  positive definite $2 \times 2$ matrix  function on $\overline \Om$  and 
$\qc$  a non--negative smooth function.

Here, let $u_h(t)\in S_h$,  denote 
the standard Galerkin finite element approximation of $u(t)$, defined by
\begin{equation}\label{1.fem-general}
 (u_{h,t}, \chi) + a( u_h, \chi) =0,\quad \forall \chi\in S_h, \for t
\ge 0, \with u_h(0)=v_h,
\end{equation}
where $v_h\in S_h$ is an approximation of $v$ and 
\begin{equation}\label{a:form}
{
a( w, \varphi) =(\all \nabla w,\nabla \varphi)
             + (\qc w,  \varphi), \for  w, \varphi  \in H^1_0. 
}
\end{equation}
In a straight--forward  way the estimates \eqref{1.sm}--\eqref{1.hsm}  
extend to  the solution of \eqref{1.fem-general}.

The natural generalization of the finite volume method
\eqref{fv} would now be to find $\tu_h(t)\in S_h$  such that
\begin{equation}\label{1.fv-general-2}
 \lla \tu_{h,t}, \chi \rra + a_h(\tu_h, \I_h\chi) =0, 
\quad \forall \chi \in S_h,\for
 t\ge0, \with  \tu_h(0)=v_h,
\end{equation}
where, instead of \eqref{a-h0}, one uses the bilinear 
defined by
\begin{equation}\label{1.a-h}
a_h(\vv,\eta)=\sum_{z\in Z_h^0} \eta(z)
     \Big ( - \int_{\partial V_z}  (\all \nabla \vv) \cdot n \, d\sigma  + \int_{ V_z} 
\qc \vv \,dx \Big ), \ \forall\vv \in S_h,\, \eta\in Y_h.
\end{equation}
It is known that, in general, the bilinear form 
$a_h(\psi,J_h\chi)$,  is nonsymmetric on $S_h$ but it is not far from being symmetric, or 
$|a_h(\chi,\I_h\psi)-a_h(\psi,\I_h\chi)|\le Ch\|\nabla \chi\|\,\|\nabla\psi\|$,
cf. \cite{cl00}. Also, if $\all$ and $\qc$ are constants  over each 
$\K\in \T_h$, then, see, e.g. \cite{Bank-rose,huang-xi},
\begin{equation}\label{a-symmetry}
a_h(\psi, \I_h\chi)=(\all \nabla \psi, \nabla \chi) + 
(\qc  \psi, \I_h\chi),\quad\forall \psi,\chi\in S_h,
\end{equation} 
and thus $a_h(\psi,\I_h\chi)$ is symmetric, since as we shall show
 $(\qc  \psi, \I_h\chi)=(\qc  \chi, \I_h\psi)$. Therefore, since 
symmetry is important in our analysis, we introduce the modified 
bilinear form
\begin{equation}\label{1.a-h-modified}
\tah(\vv,\eta)=\sum_{z\in Z_h^0} \eta(z)
     \Big ( - \int_{\partial V_z}  (\alb \nabla \vv) \cdot n \, d\sigma  
+ \int_{ V_z} \qbar \vv \,dx \Big ),
\quad \forall\vv \in S_h,\ \eta\in Y_h,
\end{equation}
where, for $ z \in \K$, $ \K \in \T_h$, 
$\alb(z)=\all(z_\K)$ and  $\qbar(z)=  \qc(z_\K)$, with
$z_\K$ the barycenter of $\K$. 
This choice of
$\tah(\cdot,\cdot)$ leads to the  finite volume element
method, to find $\tu_h(t) \in S_h $ such that
\begin{equation}\label{1.fv-general-22}
 \lla \tu_{h,t}, \chi \rra + \tah(\tu_h, \I_h\chi) =0, 
\quad \forall \chi \in S_h,\for t\ge0, \with  \tu_h(0)=v_h,
\end{equation}
and for this the desired analogues of the estimates 
\eqref{lm-estimates-1} are established in Theorems
\ref{gen-smooth}--\ref{gen-sym}.

The following is an outline of the paper. In Section
\ref{sec:prelim}, we introduce notation and give some preliminary
material needed for the analysis of the finite volume element method.
Further, we derive smooth and nonsmooth initial data  estimates for
the gradient of the error in the standard Galerkin method. In Section 
\ref{sec:smooth} we derive the error estimates \eqref{lm-estimates-1} 
discussed above, under the different assumptions on smoothness of data and 
the triangulations $\{\T_h\}$. In Section \ref{sec:specialmeshes} we show that assumption
\eqref{higher_error} is valid for symmetric meshes, and discuss
the corresponding properties for almost symmetric and piecewise almost symmetric meshes. 
In Section \ref{sec:counterexample} we present two nonsymmetric triangulations
in two space dimensions for which optimal order $L_2$--convergence for
nonsmooth data does not hold. In Section \ref{fully} we
consider briefly the application 
to the fully discrete backward Euler and
Crank--Nicolson  finite volume  methods. Finally, Section
\ref{sec:general} contains the  extension
 of Section \ref{sec:smooth}
to more general parabolic equations.


\section{Preliminaries}\label{sec:prelim}


In this section we show a smoothing property for the finite volume element method, and discuss 
the quadrature associated with this method. We also derive some estimates for the gradient of 
the error in the standard Galerkin finite element method which will be needed later.

We first recall that for the standard Galerkin method, one may introduce the discrete 
Laplacian $\De_h: S_h\to S_h$ by
\[
-(\De_h\psi,\chi)=(\nabla
\psi,\nabla\chi),\quad\forall\psi,\chi\in S_h,
\]
and write the  problem \eqref{fem}  as
\begin{equation}\label{fem-operator}
u_{h,t}-\De_hu_h=0, \for t\ge0,\with u_h(0)=v_h.
\end{equation}
Letting $\{\la^h_j\}_{j=1}^{N_h}$, $\{\phi_j^h\}_{j=1}^{N_h}$, where
$N_h=\dim S_h$, denote
the eigenvalues, in increasing order, and the corresponding
 eigenfunctions of $-\De_h$, orthonormal with respect to $(\cdot,
\cdot)$, we have for the solution operator
$E_h(t)=e^{\De_h t}$ of \eqref{fem-operator}, by
eigenfunction expansion,
\begin{equation*}
u_h(t)=E_h(t)v_h=\sum_{j=1}^{N_h}e^{-\la^h_jt}(v_h,\phi_j^h)\phi_j^h,
 \for t\ge0 .
\end{equation*}
The following smoothing property analogous to \eqref{1.smooth}
holds for $v_h\in S_h$ and $t>0$
\begin{equation}\label{fem-reg}
\|\nabla^pD_t^\ell E_{h}(t)v_h\|\le Ct^{-\ell-(p-q)/2}\|\nabla^q
v_h\|,\quad \ell\ge0,\ p,q=0,1,\ 2\ell+p\ge q,
\end{equation}
with $D_t=\partial/\partial t$.

Turning to the finite volume  method \eqref{vvv}, we now introduce the
discrete Laplacian $\wtDe_h:\,S_h\to S_h,$ corresponding to the
inner product $\lla\cdot,\cdot\rra$ in \eqref{1.fvip}, by
\begin{equation}\label{bDe_h-def}
-\lla\wtDe_h\psi,\chi\rra=(\nabla\psi,\nabla\chi),
\quad\forall\psi,\chi\in S_h.
\end{equation}
The finite volume method \eqref{vvv} can then be written in operator
form as
\begin{equation}\label{lumped}
\tu_{h,t}-\wtDe_h\tu_h=0, \for t \ge0, \with \tu_h(0)=v_h.
\end{equation}
For the solution operator $\wtE_h(t)=e^{\wtDe_h t}$  of
\eqref{lumped} we have
\begin{equation}\label{bEh}
\tu_h(t)=\wtE_h(t)v_h=\sum_{j=1}^{N_h}e^{-\tlj t}\lla
v_h,\tpj\rra\tpj, \for t\ge0,
\end{equation}
where $\{\tlj\}_{j=1}^{N_h}$ and $\{\tpj\}_{j=1}^{N_h}$ are the
eigenvalues, in increasing order, and the corresponding
eigenfunctions, orthonormal with
respect to $\lla\cdot,\cdot\rra$, of the positive definite operator
$-\wtDe_h$. For $\wtE_h(t)$  the following analogue of \eqref{fem-reg}
holds, cf. \cite[Lemma 2.1]{clt11}.
\begin{lemma}\label{bEh-estimates} For  $\wtE_h$
defined by \eqref{bEh} we have, for $v_h\in S_h$ and $t >0$
\[
\|\nabla^pD_t^\ell \wtE_{h}(t)v_h\|\le Ct^{-\ell-(p-q)/2}\|\nabla^q
v_h\|, \, 
\ell\ge0,\ p,q=0,1,\ 2\ell+p\ge q.
\]
\end{lemma}

\begin{proof} Introducing the square root $\widetilde
G_h=(-\wtDe_h)^{1/2}:S_h\to S_h$, of $-\wtDe_h$, we get
\[
\|\nabla v_h\|^2=\lla(-\wtDe_h)v_h,v_h\rra=\sum_{j=1}^{N_h}\tlj \lla v_h,\tpj\rra^2
=\tribar\widetilde G_h
v_h\tribar^2 .
\]
Since the norms $\tribar\cdot\tribar$ and $\|\cdot\|$ are equivalent
on $S_h$ we find for $t>0$
\begin{align*}
\ \ \|\nabla^pD_t^\ell  \wtE_{h}(t)v_h\|^2 &\le C\tribar\widetilde
G_h^pD_t^\ell \wtE_{h}(t)v_h\tribar^2
=C\sum_{j=1}^{N_h}(\tlj)^{2\ell+p-q}e^{-2\tlj t} (\tlj)^q\lla
v_h,\tpj\rra^2
\\
&\le C\,t^{-(2\ell+p-q)}\tribar\widetilde G_h^q v_h\tribar^2 \le
C\,t^{-(2\ell+p-q)}\|\nabla^q v_h\|^2. 
\qedhere
\end{align*}
\end{proof}
The quadrature error functional $\vep_h(\cdot,\cdot)$ defined by
\eqref{veph-def} has an important 
role in our analysis below. For this reason we recall the following 
lemma, cf. \cite{chatzipa-l-thomee04}.
\begin{lemma}\label{quad-error-lemma} For  the
error functional $\vep_h$, defined by \eqref{veph-def}, we have 
\begin{equation*}
|\vep_h(f, \psi)| \le C h^{p+q} \| \nabla^p f\|\,
\|\nabla^q\psi\|,\quad \forall 
f\in H^1,\ \psi\in S_h,  \andy\,  p,q=0,1.
\end{equation*}
\end{lemma}
\begin{proof}
Since $\int_\tau(J_h\psi-\psi)\,dx=0$ for $\psi$ linear in $\tau$,
for any $\tau\in\T_h$, see \cite{cl00},
we have that $J_h\psi-\psi$ is orthogonal to $\bar S_h$, the set
of piecewise constants on $\T_h$. Hence
\[
\vep_h(f,\psi)=(f,J_h\psi-\psi)=(f-\bar P_hf,J_h\psi-\psi),
\]
where $\bar P_h$ is the orthogonal projection onto $\bar S_h$.
The lemma now easily follows since
$\|J_h\psi-\psi\|\le Ch\|\nabla\psi\|$ and
$\|\bar P_hf-f\|\le Ch\|\nabla f\|$.
\end{proof}

The following  estimate holds for the quadrature error operator
$Q_h$ in \eqref{q_h-def}.
\begin{lemma}\label{q_h-stability}
 Let $\wtDe_h$ and $Q_h$ be the operators defined by
\eqref{bDe_h-def} and \eqref{q_h-def}. Then
\begin{equation*}
\|\nabla Q_h\chi\|+h\|\wtDe_h Q_h\chi\|\le
Ch^{p+1}\|\nabla^p \chi\|,\quad\forall\chi\in S_h,\  p=0,1.
\end{equation*}
\end{lemma}
\begin{proof}
By \eqref{q_h-def} and Lemma \ref{quad-error-lemma}, with $\psi=Q_h\chi$ and
$q=1$,    it follows easily  that
 \begin{equation*}\label{q_h-estimates}
\|\nabla Q_h\chi\|^2 =\vep_h(\chi,Q_h\chi)\le
Ch^{p+1}\|\nabla^p\chi\|\, \|\nabla Q_h\chi\|, \for p=0,1,
 \end{equation*}
which shows the desired estimate for $\|\nabla Q_h\chi\|$. Also, by the
definition of $\wtDe_h$,  Lemma \ref{quad-error-lemma} with $q=0$ shows, for $p=0,1,$
 \[\tribar\wtDe_h Q_h\chi\tribar^2
=-(\nabla Q_h\chi,\nabla\wtDe_h Q_h\chi)=-\vep_h(\chi,\wtDe_h Q_h\chi)\le
Ch^p\|\nabla^p \chi\|\,\|\wtDe_h Q_h\chi\|.
\]
Since the norms $\tribar\cdot\tribar$ and 
$\|\cdot\|$ are equivalent
on $S_h$, this implies the  bound for remaining term $\|\wtDe_h Q_h\chi\|$.
\end{proof}

In addition to the orthogonal $L_2$--projection $P_h$, 
our error analysis will use the Ritz projection $R_h:H^1_0\to S_h$,
defined by
\[
(\nabla R_hw,\nabla\chi)=(\nabla w,\nabla\chi),\quad \forall \chi\in
S_h.
\]
It is well--known
that $R_h$ satisfies
\begin{equation}\label{rh-bound}
\|R_h w-w\|+h\|\nabla(R_hw-w)\|\le Ch^q| w|_q, \for w\in \dot H^q,\
q=1,2.
\end{equation}

We close with some estimates for the gradient of the
error, slightly generalizing those of \cite[Theorem 2.1]{clt11}.

\begin{theorem}\label{SG-H1}
Let $u$  and $u_h$ be the solutions of \eqref{eq1} and \eqref{fem-operator}. Then, for $t>0$,
\[
\|\nabla(u_h(t) - u(t))\| \le
\begin{cases}
Ch|v|_2,  &\ifff\ \|\nabla(v_h-v)\|\le Ch|v|_2,
\\
Cht^{-1/2}|v|_1,\  &\ifff\ \|v_h-v\|\le Ch|v|_1,
\\
Cht^{-1}\|v\|,\  &\ifff\ v_h=P_hv.
\end{cases}
\]
\end{theorem}
\begin{proof} 
In \cite[Theorem 2.1]{clt11} this was shown with $v_h=R_h v$ in the first two estimates, 
and thus  it remains to bound $\nabla E_h(t)(v_h -R_h v)$. With
$\vth:=v_h -R_h v$ we find easily,
by Lemma \ref{bEh-estimates}, for smooth data, 
 $\|\nabla E_h(t)\vth(0)\|\le\|\nabla\vth(0)\|
\le Ch|v|_2$, and for mildly nonsmooth data, 
 $\|\nabla E_h(t)\vth(0)\|\le Ct^{-1/2}\|\vth(0)\| \le Ct^{-1/2}h|v|_1.$
\end{proof}


\section{Smooth and nonsmooth initial data error estimates}
\label{sec:smooth}

In this section we derive optimal order error estimates for the
finite volume element method \eqref{vvv}, with initial data $v$ in
$\Hdot^2$, $\Hdot^1$ and $L_2$. For $v\in \Hdot^2$, the error estimate
is the same as that for  the standard Galerkin finite element method,
and this is also the case for $v\in \dot H^1$, provided
the family of finite element spaces is quasi--uniform.
In the case  $v\in L_2$, with discrete initial data $v_h=P_hv$, in order 
to derive an optimal order estimate analogous to \eqref{1.nsm},
we need to impose  condition \eqref{higher_error} 
for  the quadrature error
operator $Q_h$. In Section \ref{sec:specialmeshes} we  verify this condition for symmetric 
meshes.
In the general case we are only able to show a
non--optimal order $O(h)$ error bound in $L_2$, whereas for the
gradient of the error an optimal order $O(h)$ bound still holds.

The estimates and their proofs are analogous to those for the lumped mass
method derived in \cite{clt11}, since the operators $\wtE_h$,
$\wtDe_h$ and $Q_h$, defined in Section \ref{sec:prelim}, have
properties similar to those of the corresponding operators for the lumped mass 
method. References to \cite{clt11} will therefore be given  in some of the proofs below.
We begin with smooth initial data, $v\in \dot H^2$.


\begin{theorem}\label{lumped-L2-norm-smooth}
Let $u$  and $\tu_h$  be the solutions of \eqref{eq1} and 
\eqref{lumped}. Then 
\[
\| \tu_h(t)-u(t)\| 
\le Ch^2|v|_2, ~~~\ifff\ \|v_h-v\|\le Ch^2|v|_2, \for\, t \ge 0.
\]
\end{theorem}

\begin{proof}
Since, by \eqref{1.sm}, the corresponding error bound holds for the solution $u_h$ of
the standard Galerkin method,  
it suffices to consider the difference $\de=\tu_h-u_h $. 
Also, by the stability estimates of Lemma \ref{bEh-estimates}, we may assume that $v_h=R_hv$.
By the definition \eqref{q_h-def} of $ Q_h$, 
$\de$  satisfies \eqref{1.del}, and hence
\begin{equation}\label{delta-eq}
\de_t-\wtDe_h\de=\wtDe_hQ_hu_{h,t}, \for t\ge0, \with \de(0)=0,
\end{equation}
where $u_h$ is the solution of \eqref{fem}. By Duhamel's principle this shows
\begin{equation}\label{delta-sol}
\delta(t)=\int_0^t\wtE_h(t-s)\wtDe_hQ_hu_{h,t}(s)\,ds.
\end{equation}
Using the fact that $\wtE_h(t)\wtDe_h=D_t\wtE_h(t)$, and Lemmas
\ref{bEh-estimates} and \ref{q_h-stability}, we easily get
\begin{equation}
\begin{split}
\label{bEh-bDeh-Qh}
\|\wtE_h(t) \wtDe_hQ_h\chi\| 
 \le
& Ct^{-1/2} \|\nabla Q_h\chi\| 
\le Ch^2t^{-1/2}\|\nabla\chi\|,\for \chi\in S_h,
\end{split}
\end{equation}
and hence
\[
\|\de(t)\|
\le Ch^2\int_0^t(t-s)^{-1/2} \|\nabla u_{h,t}(s)\| \,ds.
\]
Here, since $\De_hR_h=P_h\De$, we obtain, by first applying Lemma
\ref{bEh-estimates},
\[
\|\nabla u_{h,t}(s)\| \le Cs^{-1/2} \| u_{h,t}(0)\|= Cs^{-1/2} \| \De_h R_hv\| 
\le Cs^{-1/2} \| \De v\|
= Cs^{-1/2} |v|_2,
\]
and hence
\[
\|\de(t)\| 
\le Ch^2\int_0^t(t-s)^{-1/2} s^{-1/2}\,ds \,|v|_2 =C\,h^2|v|_2,
\]
which completes  the proof.
\end{proof}

We now consider mildly nonsmooth initial data, $v\in \dot H^1$.
Here we shall need to assume the stability of $P_h$ in $\dot H^1$, or 
$\|\nabla P_h w\|\le C|w|_1 $, which does not hold for arbitrary  families of triangulations. 
However, a  sufficient 
condition  for such  stability of $P_h$ is the global quasi--uniformity   of $\{\T_h\}$.
Indeed,  this assumption implies  the inverse inequality
$\| \nabla \chi \| \le C h^{-1} \|\chi\|$, which combined with the error bound 
$\|R_h w- w\|\le Ch| w |_1$, shows  the desired stability of $P_h$.

\begin{theorem}\label{lumped-L2-norm-msmooth}
Let $u$ and $\tu_h$  be the solutions of \eqref{eq1} and 
\eqref{lumped}. Then for $t>0$
\[
\| \tu_h(t)-u(t)\|\le Ch^2t^{-1/2}|v|_1,  \ \ifff v_h=P_hv \andy\  
\|\nabla P_hv\|\le C|v|_1. 
\]
\end{theorem}
\begin{proof} 
Since  by  \eqref{1.hsm}, the corresponding error estimate holds for the solution 
$u_h$ of the standard Galerkin method (without the condition on $\nabla P_h$),  
it suffices as above to bound $ \de= \tu_h-u_h$. We use \eqref{delta-sol} to write
\begin{equation}
\label{delta-sol-2} \de(t) =\Bigl\{\int_0^{t/2}+
     \int_{t/2}^t\Bigr\}\wtE_h(t-s)\wtDe_hQ_hu_{h,t}(s)\,ds
     =\de_1(t)
     +\de_2(t).
\end{equation}
Using again \eqref{bEh-bDeh-Qh}, we have, since 
$\|\nabla u_{h,t}(s)\| \le C  s^{-1}\|\nabla P_hv\| \le C s^{-1}\,|v|_1 $, that
\begin{equation*}
\|\de_2(t)\| 
\le Ch^2\int_{t/2}^t (t-s)^{-1/2} \,\|\nabla u_{h,t}(s)\| \,ds
\le Ch^2t^{-1/2}\,|v|_1.
\end{equation*}
Integrating by parts,  we obtain
\begin{equation}\label{term-I}
\de_1(t)=\Bigl[\wtE_h(t-s)\wtDe_hQ_hu_{h}(s)\Bigr]_{0}^{t/2}
     -\int_0^{t/2}D_s\wtE_h(t-s)\wtDe_hQ_hu_{h}(s)\,ds.
\end{equation}
Employing \eqref{bEh-bDeh-Qh}, Lemmas \ref{bEh-estimates} and \ref{q_h-stability}
 we now find, similarly to the above,
\begin{align*}
\|\de_1(t)&\| 
\le  Ch^2t^{-1/2}(\|\nabla u_h(t/2)\|+ \|\nabla P_hv\|)
\\
&+Ch^2\int_{0}^{t/2}(t-s)^{-3/2}\,\|\nabla u_h(s)\|\,ds
  \le Ch^2t^{-1/2}|v|_1.
\end{align*}
Together these estimates complete the proof. 
\end{proof}

The analogous result and its proof also hold for the lumped mass method, which should replace
the case $q=1$ in \cite[Theorem 3.1]{clt11}, since \eqref{1.hsm} does not hold for 
$v_h=R_hv$.

Next, we turn to the nonsmooth initial data error estimate.
\begin{theorem}\label{lumped-nonsmooth-opt}
Let $u$ and $\tu_h$  be the solutions of \eqref{eq1} and  
\eqref{lumped}. If \eqref{higher_error} holds and $v_h =P_hv$, then 
\[
\| \tu_h(t) -  u(t)\| \le Ch^2t^{-1}\| v \|, \for
t>0.
\]
\end{theorem}
\begin{proof} This  follows easily from the fact that for 
$Q_h$ satisfying  \eqref{higher_error} we have, 
\begin{equation}\label{opt-assumption}
\|\wtE_h(t)\wtDe_hQ_h P_h v\|\le Ct^{-1}\|Q_h P_h v\|\le
Ch^2t^{-1}\|v\|,  \for t>0.
\end{equation}
This inequality is the necessary and sufficient condition for 
desired bound 
to hold by the 
following lemma, which is proved in the same way as
\cite[Theorem 4.1]{clt11}.
\end{proof}
\begin{lemma}
\label{lumped-L2-norm-nonsmooth} Let $u$ and  $\tu_h$  be the solutions of
\eqref{eq1} and 
\eqref{lumped}.  Then 
\begin{equation*}
\| \tu_h(t) -  u(t)+\wtE_h(t)\wtDe_hQ_h v_h\| \le Ch^2t^{-1}\| v
\|,\  \ifff v_h =P_h v, \for  t>0,.
\end{equation*}
\end{lemma}
Condition \eqref{higher_error} will be discussed in more detail
in Section  \ref{sec:specialmeshes} below. Note that, by Lemma
\ref{q_h-stability}, without additional assumptions on the mesh, we
 have
\[\| Q_h \chi \| \le
C\|\nabla  Q_h \chi \| \le C h \|\chi\|,  \quad  \forall \chi\in
S_h,
\]
and that the lower order error estimate of the following theorem always holds.
The proof is the same as that of \cite[Theorem 4.3]{clt11}. We shall show
in Section \ref{sec:counterexample} that a $O(h)$ bound is the best possible for 
general triangulation families $\{\T_h\}$.

\begin{theorem}\label{t4.3old}
 Let $u$ and $\tu_h$  be the solutions of \eqref{eq1} and 
\eqref{lumped}. Then 
\begin{equation*}\label{lumped-nonsmooth-estimate-nonopt}
\| \tu_h(t) -  u(t)\| \le Cht^{-1/2}\| v \|,\quad \ifff v_h = P_h v, \for t>0.
\end{equation*}
\end{theorem}
We end this section by stating optimal order estimates for the
gradient of the error. Note that no additional assumption on $\{\T_h\}$ is required.
\begin{theorem}
\label{t4.3}
 Let $u$ and $\tu_h$  be the solutions of \eqref{eq1}
and \eqref{lumped}. Then, for $t>0$,
\[
\|\nabla(\tu_h(t) - u(t))\| \le
\begin{cases}
Ch|v|_2,  &\ifff\ \|\nabla(v_h-v)\|\le Ch|v|_2,
\\
Cht^{-1/2}|v|_1,\  &\ifff\ \|v_h-v\|\le Ch|v|_1,
\\
Cht^{-1}\|v\|,\  &\ifff\ v_h=P_hv.
\end{cases}
\]
\end{theorem}
\begin{proof}
 For the first two estimates it suffices, by the stability and smoothness
estimates of Lemma \ref{bEh-estimates}, to consider $v_h = R_h v$. For this choice of the
initial data the proofs are identical to those in \cite[Theorem 3.1]{clt11}. In the 
nonsmooth data case, the proof is the same as that of \cite[Theorem 4.4]{clt11}.
\end{proof}


\section{Symmetric and almost symmetric triangulations} 
\label{sec:specialmeshes}


In this section we first show that for families of triangulations $\{\T_h\}$ that 
are {\it symmetric}, in a sense to be defined below,  assumption \eqref{higher_error} is
satisfied and therefore, by Theorem \ref{lumped-nonsmooth-opt}, the optimal order 
nonsmooth data error estimate holds. We shall then 
relax the symmetry requirements and consider {\it almost symmetric} families of 
triangulations, consisting of $O(h^2)$ perturbations of symmetric triangulations. 
In this case we show that \eqref{higher_error} is satisfied with  an additional 
logarithmic factor and, as a consequence, an almost optimal order nonsmooth data error 
estimate holds. Finally for the less restrictive class of {\it piecewise almost symmetric} 
families $\{\T_h\}$ we derive a $O(h^{3/2})$ order nonsmooth data error estimate.

In addition to the quadrature error
operator $Q_h$ defined in \eqref{veph-def} we shall work
with the symmetric operator
$M_h:S_h\to S_h$, defined by
\begin{equation}\label{Mh-def0}
\vep_h( \w, \chi)= [\psi,M_h\chi],
\quad\forall \psi,\chi\in S_h, 
\end{equation}
where we use the inner product 
\begin{equation}\label{4.ipdef}
 [\psi,\chi]=
\sum_{z \in Z_h^0}\w(z)\chi(z), 
\quad\forall \psi,\chi\in S_h.
\end{equation}
To determine the form of this operator, we introduce some notation.
For $z\in Z_h^0$ an interior vertex of   $\T_h$, 
we define the patch
$\Pi_z = \{ \cup \K: \tau \in \T_h, \, z\in \partial \K  \}$, 
 where for 
simplicity we have assumed that $\tau=\bar\tau$. Further, for $z$ a
vertex of $\K\in \T_h$, we
 denote by $z^{\K}_{+}$ and $z^{\K}_{-}$ the other two vertices of
$\K$. We
then define,
\begin{equation}\label{4.delta}
 M_h^{\Pi_z} \chi 
:= -\frac{1}{54}\sum_{\K \subset \Pi_z} |\K| ( \chi(z^{\K}_{+})-2
\chi(z) +\chi(z^{\K}_{-})),
\end{equation}
for which the following holds.

\begin{lemma}\label{lemma-Mh-def0}
 For the operator $M_h$ defined by
\eqref{Mh-def0} we have, for $z\in Z_h^0$,
\begin{equation}\label{Mh-definition}
 M_h \chi(z) 
= M_h^{\Pi_z} \chi 
\with M_h^{\Pi_z} \chi 
\text{ given by \eqref{4.delta}}.
\end{equation}
\end{lemma}
\begin{proof}
 In view of \eqref{veph-def}, we may write
\begin{equation}\label{4.4prime}
 \vep_h( \w, \chi)= (\w,
\I_h\chi)-(\w, \chi) = \sum_{\K \in \T_h } \int_\K (
\w\I_h\chi-\w \chi)\,dx.
\end{equation}
\begin{figure}[t]
\centering{
\includegraphics[width=0.28\textwidth]{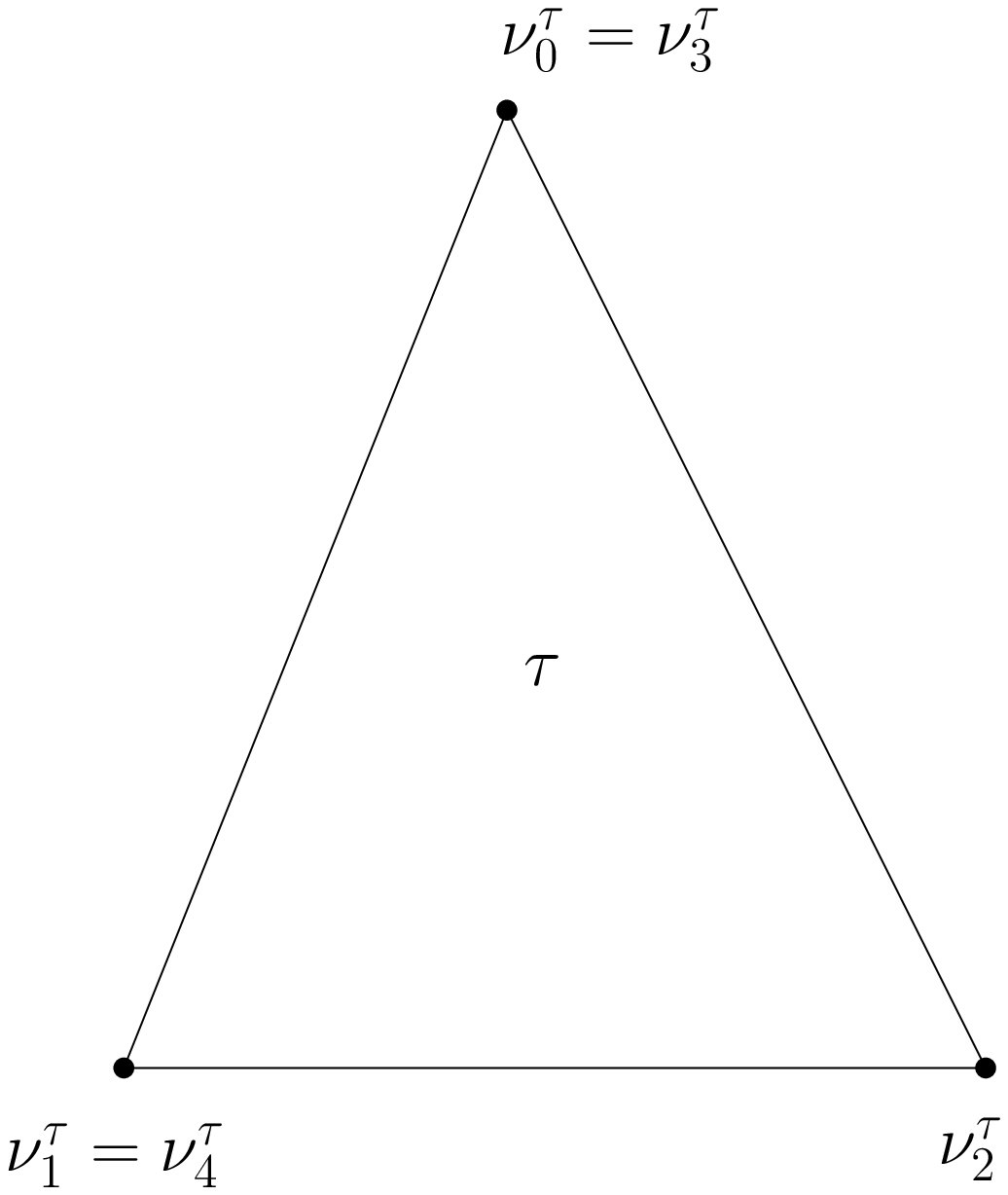}
\hspace{1cm}
\includegraphics[width=0.3\textwidth]{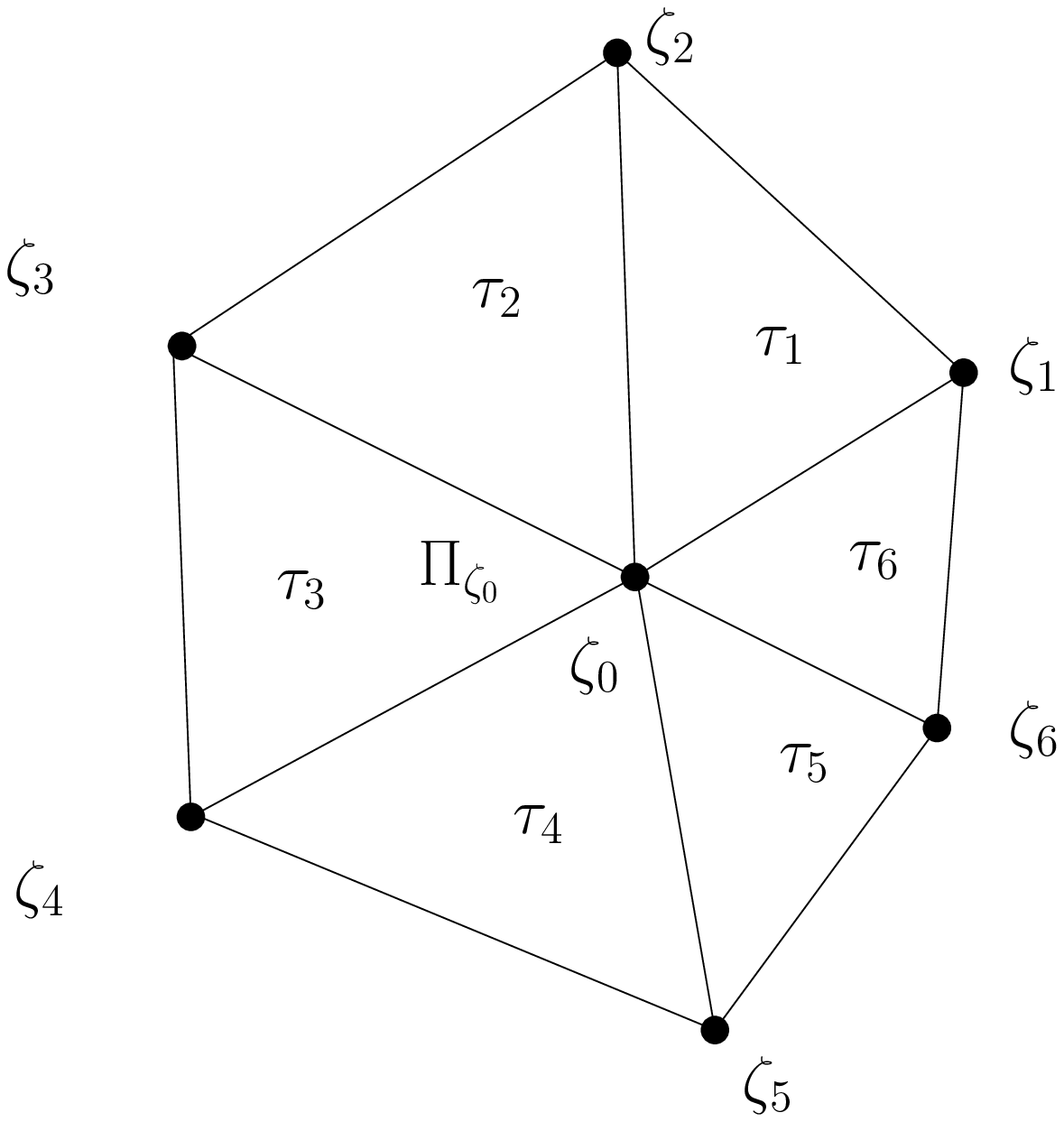}
}
\caption{Left: A triangle $\tau$. Right: A patch $\Pio0$ around a
vertex
$\z_0$} \label{fig:tri}
\end{figure}
For $\K\in \T_h$ we denote its vertices by
$\nu_1^\tau,\nu_2^\tau,\nu_3^\tau$ and set
$\nu_4^\tau=\nu_1^\tau$, $\nu_0^\tau=\nu_3^\tau$,
 see Figure \ref{fig:tri}.
Writing $w_j=w(\nu_j^\tau)$ for a function $w$ on $\tau$, we obtain,
after  simple calculations,
\begin{equation}\label{4.5}
\int_\K \psi\I_h\chi\, dx
=\dfrac{|\K|}{108}\sum_{j=1}^3\psi_j(22\chi_j+7\chi_{j-1}+7\chi_{j+1}),
\end{equation}
and
\[
\int_\K\psi\chi\,dx
=\dfrac{|\K|}{12}\sum_{j=1}^3\psi_j(2\chi_j+\chi_{j-1}+\chi_{j+1}).
\]
Thus
$$
\int_\K ( \w\I_h\chi-\w \, \chi)\,dx
=-\frac{|\K|}{54}\sum_{j=1}^3 \psi_j
 (\chi_{j+1}-2\chi_j+\chi_{j-1}).
$$
Summation over $\tau\in \T_h$, \eqref{Mh-def0}  and \eqref{4.4prime} show
\[
[\psi,M_h\chi]=\sum_{z\in Z_h^0}\psi(z) M_h^{\Pi_z}\chi,
\quad\forall \psi,\chi\in S_h.
\]
This implies \eqref{Mh-definition} and thus completes the proof.
\end{proof}

\begin{figure}[t]
\centering{
\includegraphics*[width=0.7\textwidth]{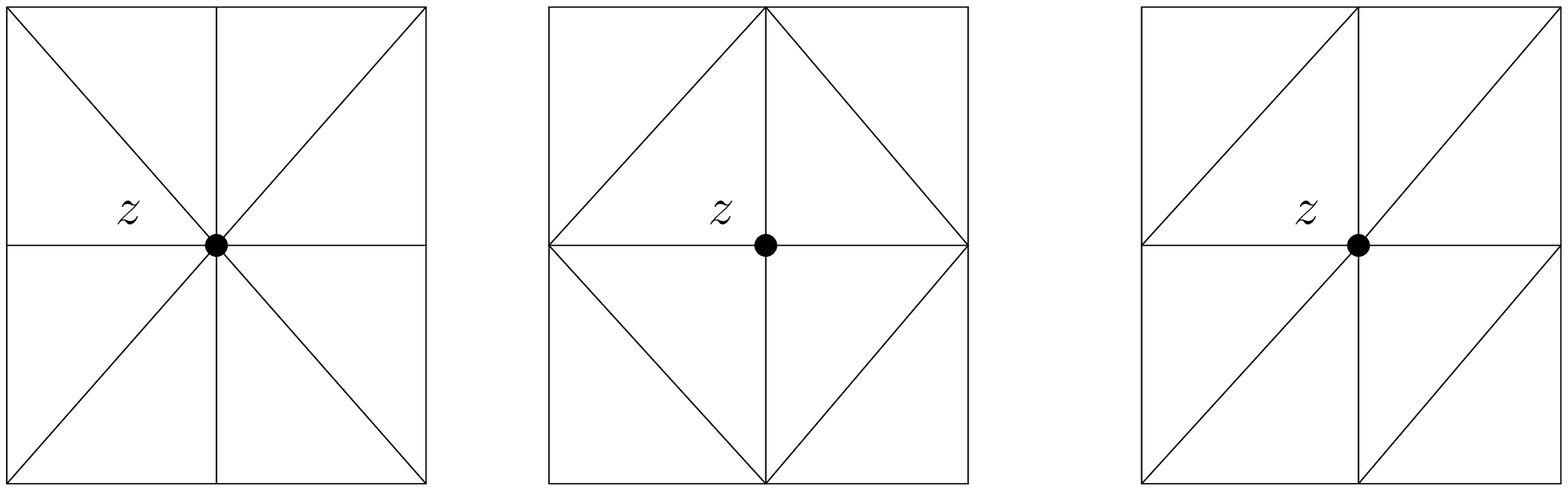}
}
\caption{Patches which are symmetric with respect to the vertex
$z$} \label{fig:s-meshes-2}
\end{figure}
We say that  $\T_h$ is {\it symmetric at $z\in Z_h^0$,} if the
corresponding patch $\Pi_z$ is symmetric  around $z$, in the sense
that if $x\in \Pi_z$, then $z-(x-z)=2z-x \in \Pi_z$.
We say that $\T_h$ is {\it symmetric} if it is symmetric at
each $z \in Z^0_h$. The patch $\Pio0$ in Figure \ref{fig:tri} is
nonsymmetric with respect to $\z_0$, whereas triangulations which are built up 
of either of the patches shown in Figure \ref{fig:s-meshes-2} are
symmetric. Symmetric triangulations exist only for special domains,
such as parallelograms, but not for general polygonal domains.

 We now show the sufficiency of symmetry of  $\{\T_h\}$ for condition
\eqref{higher_error} for the operator $Q_h$, and hence, by Theorem
\ref{lumped-nonsmooth-opt},
for the nonsmooth data error estimate.
\begin{theorem}\label{verify}
If the family $\{\T_h\}$ is symmetric, then \eqref{higher_error} holds.
\end{theorem}
\begin{proof}
The proof, by duality, follows that of \cite[Theorem 5.1]{clt11}. For given 
$\chi \in S_h$ we define $\vfy=\vfy_\chi\in \dot H^1$ as the solution of the Dirichlet
problem $-\Delta \vfy=\chi$ in $\Om$, $\vfy=0$ on $\partial\Om$. Since $\Om$ is convex, 
we have $\vfy \in \dot H^2$ and $|\vfy|_{2} \le C \|\chi\|$. With $\Id_h$  the finite 
element interpolation operator into $S_h$, we  have, for any $\psi \in S_h$,
\begin{align}
 \label{l2-norm}
\| Q_h \psi \| & =\sup_{\chi \in S_h} \frac{(Q_h \psi,
\chi)}{\|\chi\|}
  =\sup_{\chi \in S_h}
\frac{(\nabla Q_h \psi, \nabla \vfy)}{\|\chi\|}
\\
  &\le\sup_{\chi \in S_h}
 \frac{|(\nabla \ww, \nabla (\vfy - \Id_h \vfy))|}{\|\chi\|}
  +\sup_{\chi \in S_h}
\frac{|(\nabla \ww, \nabla \Id_h \vfy)|}{\|\chi\|}=I+II.
\notag
\end{align}
By the obvious error estimate for $\Id_h$ and Lemma
\ref{q_h-stability}, with $p=0$, we find
\begin{equation}\label{4.I}
|I|\le Ch  \sup_{\chi \in S_h}\frac{\|\nabla Q_h\psi\|\,|\vfy|_2}{\|\chi\|}
\le Ch^2\|\psi\|.
\end{equation}
To estimate $II$, we employ \eqref{q_h-def} and \eqref{Mh-def0} to  
rewrite the numerator 
in the form
\begin{equation}\label{4.eh}
(\nabla \ww, \nabla \Id_h \vfy)= \vep_h( \w, \Id_h \vfy)= 
[\w, M_h \Id_h\vfy]. 
\end{equation} 

To bound $M_hI _h\vfy$,  we consider an arbitrary vertex $z=\z_0\in
Z_h^0$.
Let $\Pio0$ be the corresponding patch of $\T_h$,
with vertices $\{\z_j\}_{j=1}^K,$ numbered counter--clockwise, 
with
$\z_{j+K}=\z_j$ for all $j$. Also denote by
$\{\tau_j\}_{j=1}^K$, the triangles of $\T_h$ 
in $\Pio0$, with $\tau_j$ having vertices $\z_0,\, \z_j,\,
\z_{j+1}$, and set $\tau_0=\tau_K$ (see Figure \ref{fig:tri}). Then 
Lemma \ref{lemma-Mh-def0} implies
\begin{equation}\label{Mh-def}
 M_h \Id_h \varphi(\z_0) =M_h^{\Pio0} \Id_h \varphi 
= -\frac1{54}\sum_{j=1}^K\om_j(\varphi(\z_j)-\varphi(\z_0)),
\end{equation}
with $ \om_j=|\K_{j-1}|+|\K_j| $.
By assumption, the  patch $\Pio0$ is symmetric and hence, by
\eqref{Mh-def},  we can express $M_h\Id_h \vfy(\z_0)$ as a
linear combination of terms of the form
$\vfy(\z_j)-2\vfy(\z_0)+\vfy(\z_j')$, where $\z_0$ is the midpoint of
the vertices $\z_j$ and $\z_j'$
of $\Pio0$. Hence $M_h\Id_h\vfy(\z_0)=0$ for $\vfy$ linear in
$\Pio0$ and, as in \cite{clt11}, we may apply the Bramble--Hilbert
lemma to obtain
\begin{equation}\label{5.delst}
|M_h \Id_h\vfy(\z_0) |\le C h^2 |\Pio0
|^{1/2}\|\vfy\|_{H^2(\Pio0)}
\le C h^3 \|\vfy\|_{H^2(\Pio0)}.
\end{equation}
Employing this estimate for all patches $\Pi_z$ of $\T_h$, we obtain,
for any $\psi \in S_h$,
\begin{equation}\label{eh-bound}
\begin{split}
 |[\w, M_h \Id_h\vfy]|
\le C h^3
 \sum_{z \in Z_h^0}
|\w(z)|\,
\|\vfy\|_{H^2(\Pi_z)}
 \le Ch^2
\|\psi \| \, | \vfy |_{2} \le  Ch^2 \|\psi \|\,\|\chi\|.
\end{split}\end{equation}
Hence, in view of \eqref{l2-norm} and  \eqref{4.eh}, we obtain
$|II|\le
Ch^2\|\psi\|.$
 Together with \eqref{4.I} this completes the proof.
\end{proof} 
%
\begin{figure}[t]
\centering{
\includegraphics*[width=0.45\textwidth]{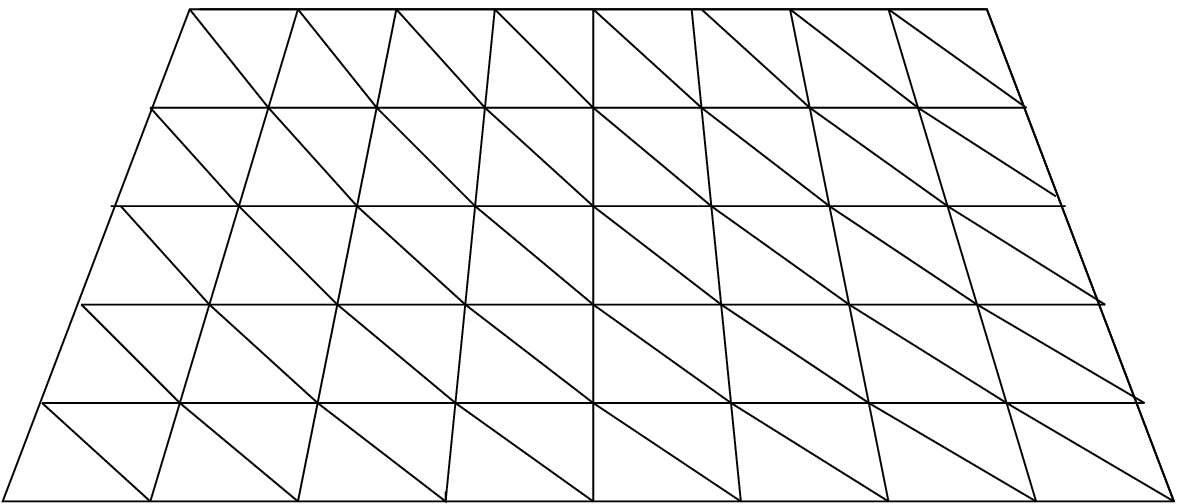}\hspace{.5cm}
\includegraphics*[width=0.45\textwidth]{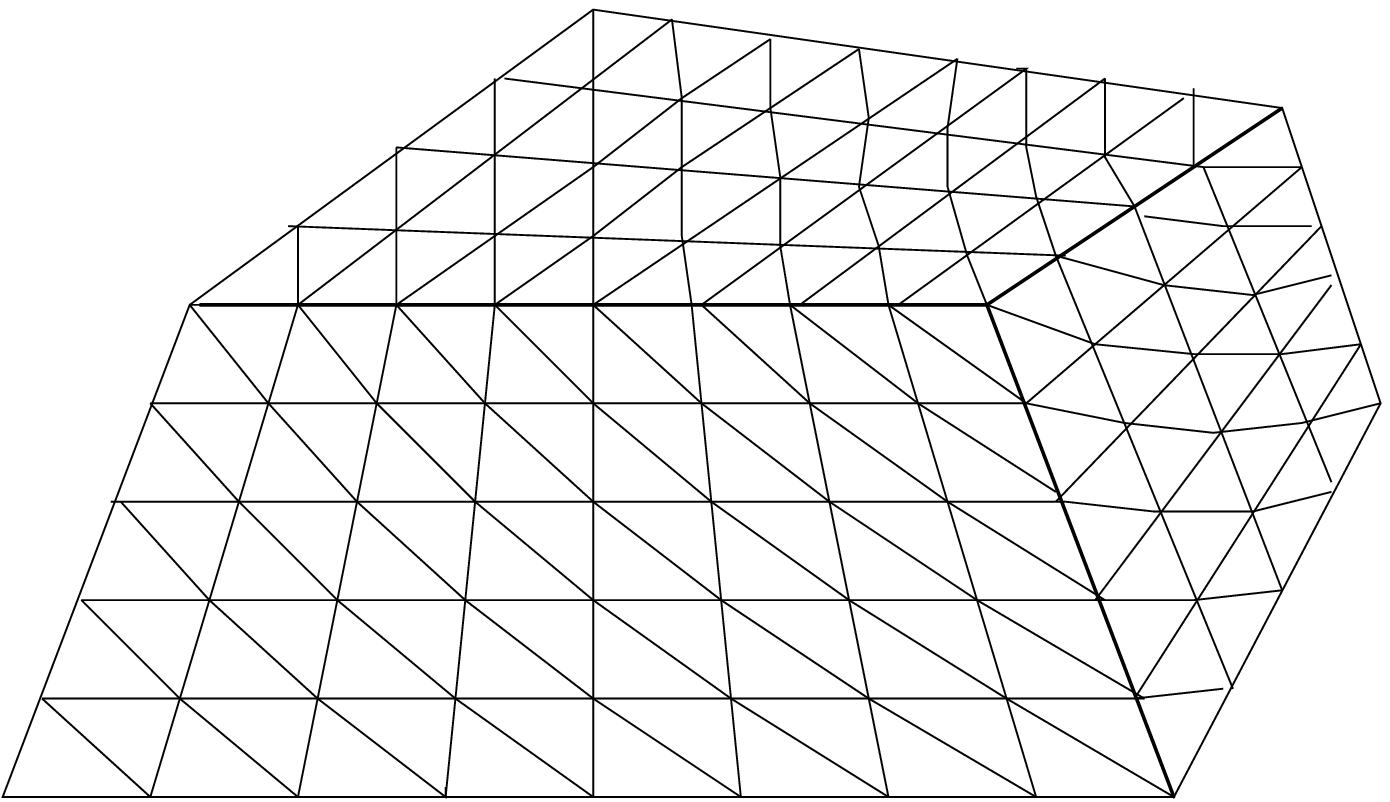}
} \caption{{\it Left}: An almost symmetric triangulation. 
 {\it Right}: A piecewise almost
symmetric triangulation.} \label{fig:piecewise-sym}
\end{figure}
We now want to slightly weaken the assumption about
symmetry.
We say that a family of triangulations $\{\T_h\}$ is  {\it almost
symmetric} if each $\T_h$ is a perturbation by $O(h^2)$ of a symmetric
triangulation, uniformly in $h$, in the sense that with each patch
$\Pi_z$
of $\T_h$ there  is an associated  symmetric patch from which $\Pi_z$
is obtained by moving each of its vertices  by $O(h^2)$. Such
triangulations exist for any convex quadrilateral, cf. 
 Figure \ref{fig:piecewise-sym}.
We  note that various special triangulations have been 
used in the past for obtaining higher order accuracy for the gradient of the
 finite element solution (super--convergent rates of $O(h^2)$ or 
$O(h^2 \ell_h)$), see, e.g., 
\cite{Krizek-Neit, Lin-Xu, Wahlbin-book}. For example, the strongly regular 
triangulations from \cite{Lin-Xu}, requiring that any two adjacent triangles 
form almost a parallelogram (a deviation of a parallelogram by $O(h^2)$), are 
almost symmetric meshes in our terminology.
We shall show that, in this case, we have almost optimal
order convergence for nonsmooth initial data.

\begin{theorem}\label{verify1}
If the  family $\{\T_h\}$ is almost symmetric, then
\begin{equation}\label{5.qlog}
\|Q_h\psi\|\le C h^{2}
\ell_h^{1/2}\,
\|\psi\|,\quad\forall\psi\in S_h,
\where \ell_h=1+|\log h|.
\end{equation}
Hence, for the solution of \eqref{vvv}, with $v_h=P_hv$, we have
\begin{equation}\label{lumped-nonsmooth-est-opt1}
\| \tu_h(t) -  u(t)\| \le Ch^2
\ell_h^{1/2}\,
t^{-1}\| v \|, \for t>0.
\end{equation}
\end{theorem} In the proof we shall need the 
following Sobolev type inequality,
where  the $|\cdot|_{H^k}$ denote 
seminorms with only the derivatives of highest order $k$.
\begin{lemma}\label{5.sob}
Let $B$ be a fixed  bounded
 domain, satisfying the cone property. Then we have, for
 $0<\ep<1$,
\[
\sup_{z,z'\in B,\,z'\ne z}\frac{|\vfy(z')-\vfy(z)|}{|z'-z|^{1-\ep}}
\le
C\ep^{-1/2}
\big(|\vfy|_{H^{1}(B)}
+|\vfy|_{H^{2}(B)}\big),
\quad\forall\vfy\in H^2(B).
\]
\end{lemma}
\begin{proof}
We find from \cite[pp. 109--110]{Adams}, for $\ep$ small,
with $C$ independent of $\ep$,
\begin{equation}\label{4.eps}
\sup_{z,z'\in B,\,z'\ne z}\frac{|\vfy(z')-\vfy(z)|}{|z'-z|^{1-\ep}}
\le C\|\nabla\vfy\|_{L_p(B)}, \with p=2/\ep,\quad\forall\,\vfy\in
W_p^1(B).
\end{equation}
We shall also apply the Sobolev inequality,
with explicit dependence on $p$,
\begin{equation}\label{4.sob}
\|\vfy\|_{L_p(B)}\le C\,p^{1/2}\|\vfy\|_{H^1(B)}, \for p<\infty,
\quad\forall\, \vfy\in H^1(B).
\end{equation}
For $\vfy \in H_0^1(B)$ a proof was sketched in \cite[Lemma
6.4]{Thomee06}.
For the general case of $\vfy \in H^1(B)$, we make a bounded extension
of $ \vfy$
from $H^1(B)$ to $H_0^1(\wt B)$, with $\wt B\subset\overline B$,
cf., \cite[IV]{Adams} and apply \eqref{4.sob} to $H^1(\wt B)$
to complete the proof.

Employing \eqref{4.sob} yields
\[
\|\nabla\vfy\|_{L_p(B)}\le C\,p^{1/2}\big(|\vfy|_{H^1(B)}
+|\vfy|_{H^2(B)}\big),\quad \forall \vfy\in H^2(B).
\]
 Combining this with \eqref{4.eps}, using $p^{1/2}=(2/\ep)^{1/2}$,
 completes the proof.
\end{proof}

\begin{proof}[Proof of Theorem \ref{verify1}]
The proof proceeds as that of Theorem \ref{verify}, starting with
\eqref{l2-norm} and noting that the bound \eqref{4.I}  for $I$
remains valid.
In order to bound $II$, we follow the steps above, but now, instead of
\eqref{5.delst}, we show 
\begin{equation}\label{4.almsym}
|M_h \Id_h\vfy(\z_0)|\le
Ch^3\ell_h^{1/2}
 \|\vfy\|_{H^2(\Pio0)}.
\end{equation}

Using \eqref{4.almsym} as \eqref{5.delst} in \eqref{eh-bound}, we find
\begin{equation}\label{4.17}
|[\w, M_h \Id_h\vfy]|\le Ch^2\ell_h^{1/2}\|\psi\|\,\|\chi\|,
\quad\forall\psi,\chi\in S_h,
\end{equation}
and hence $|II|\le Ch^2\ell_h^{1/2}\|\psi\|.$ Together with \eqref{4.I}, this completes 
the proof of \eqref{5.qlog}.
The error estimate \eqref{lumped-nonsmooth-est-opt1} now follows from 
Lemma \ref{lumped-L2-norm-nonsmooth} and
\begin{equation*}
\|\wtE_h(t)\wtDe_hQ_hP_hv\|
\le Ct^{-1}\|Q_hP_hv\|
\le Ch^2\ell_h^{1/2}t^{-1}\|v\|, \for t>0.
\end{equation*}

It remains to show \eqref{4.almsym}. 
Let  $\widetilde \Pi_{\zeta_0'}$ be the
symmetric patch associated  with 
$\Pi_{\zeta_0}$ by the definition
of almost symmetric.  After a 
preliminary translation  of $\widetilde \Pi_{\zeta_0'}$ 
by $O(h^2)$,  we may assume 
that  $\zeta_0'=\zeta_0$.
Further, without loss of 
generality, we may assume that 
$\widetilde \Pi_{\zeta_0} \subset \Pi_{\zeta_0}$. 
In fact, if this is not the case originally, it will
be satisfied by shrinking $\widetilde \Pi_{\zeta_0}$ 
by a  suitable factor $1-ch^2$ with $c\ge0$. 
Starting with
$\widetilde \Pi_{\zeta_0}$ we may now move the 
vertices one by one  by $O(h^2)$ to obtain
$\Pi_{\zeta_0}$ in 
a finite number of steps, through a sequence
of intermediate patches $\wh\Pi_{\zeta_0} \subset \Pi_{\zeta_0}$.
 
Applying \eqref{Mh-def} we will show that for each of these
\begin{equation}\label{5.fyps}
| M_h^{\widehat\Pi_{\zeta_0}} \Id_h\vfy | \le C_\ep h^{3-\ep}
\|\vfy\|_{H^2(\Pi_{\zeta_0})}, 
\where C_\ep=C\ep^{-1/2},\ \ep>0,
\end{equation}
which implies \eqref{4.almsym},  by taking
$\ep=\ell_h^{-1}$ and  $\widehat\Pi_{\zeta_0}=\Pi_{\zeta_0}$.

Since \eqref{5.fyps} holds for the symmetric patch 
$\wt\Pi_{\zeta_0}$, by \eqref{5.delst}, 
it remains to  show that if it holds for
a given patch  $\widehat\Pi_{\zeta_0}$
then it also holds for the next  patch in the sequence.
Assuming thus that \eqref{5.fyps} 
holds for $\widehat\Pi_{\zeta_0}$, 
we consider the effect 
of moving one of its vertices, $\z_2$,
say, to $\z_2'$,
with $|\z_2'-\z_2|=O(h^2)$.

Applying Lemma \ref{5.sob} to the function $\vfy(h\cdot)$, 
with $B$  suitable, we obtain
\begin{align}\label{5.winf}
\sup_{z,z'\in\Pi_{\zeta_0},\,z'\ne
z}\frac{|\vfy(z')-\vfy(z)|}{|z'-z|^{1-\ep}}
&
 \le C_\ep h^{-1+\ep}(|\vfy|_{H^{1}(\Pi_{\zeta_0})}
+h|\vfy|_{H^{2}(\Pi_{\zeta_0})})
\\
&
\le C_\ep h^{-1+\ep}\|\vfy\|_{H^{2}(\Pi_{\zeta_0})}.
\notag
\end{align}
Moving only the vertex $\z_2$  in
$\widehat\Pi_{\zeta_0}$ 
changes 
only the triangles $\tau_1$ and $\tau_2$ and thus the 
terms corresponding to $j=1,2,3$ in \eqref{Mh-def}.
 
Letting $\tau_1'$ and $\tau_2'$ be the new
triangles,  the change in the term with $j=1$ is
then bounded,  since 
$||\tau_1'|-|\tau_1||\le Ch^3$,  by
\begin{align}
|(\om_1'-\om_1)\big(\vfy(\z_1)-\vfy(\z_0)\big)|
 &  \le C\big||\tau_1'|-|\tau_1|\big|h^{1-\ep}
\frac{|\vfy(\z_1)-\vfy(\z_0)|}{|\z_1-\z_0|^{1-\ep}} \notag \\
 &  \le C_\ep h^3\|\vfy\|_{H^{2}(\Pi_{\zeta_0})}, \notag
\end{align}
and thus by the right hand side of \eqref{5.fyps}.
The change in the term with $j=3$ is bounded in the same way. 
For $j=2$ the change is bounded by the modulus of
\begin{align}
\om_2'\big(\vfy(\z_2')-\vfy(\z_0)\big)
 & -\om_2\big(\vfy(\z_2)-\vfy(\z_0)\big) \notag   \\
 & =(\om_2'-\om_2)\big(\vfy(\z_2)-\vfy(\z_0)\big)
+\om_2'\big(\vfy(\z_2')-\vfy(\z_2)\big).  \notag 
\end{align}
The first term on the right is bounded as the terms with $j=1,3$, 
and the second is bounded, using \eqref{5.winf}, since
$|\z_2'-\z_2|\le Ch^2$, in the following way
 \[ 
|\om_2'\big(\vfy(\z_2')-\vfy(\z_2)\big)|  
    \le 
C_\ep h^2|\z_2'-\z_2|^{1-\ep} h^{-1+\ep}
\|\vfy\|_{H^{2}(\Pi_{\zeta_0})}
    \le C_\ep h^{3-\ep}\|\vfy\|_{H^{2}(\Pi_{\zeta_0})}.
\]
This 
shows that \eqref{5.fyps} remains valid after moving $\z_2$, 
which concludes the proof.
\end{proof}


More generally, 
we shall consider  families of {\it piecewise almost symmetric} 
triangulations $\{\T_h\}$, in which
 $\Om$ is partitioned into a fixed set of subdomains
  $\{\Om_k\}_{k=1}^K$, and each of these is supplied with an almost symmetric family 
$\{\T_h(\Om_k)\}$ so that
 $\T_h=\cup_{k=1}^K\T_h(\Om_k)$.  
Such families may be constructed for any convex polygonal
domain, cf.
Figure \ref{fig:piecewise-sym},
by successively refining an initial  coarse mesh, a procedure routinely 
used in  computational practice. For such meshes we show  the following result.

\begin{theorem}\label{th4.3}
If the family $\{\T_h\}$ is piecewise almost symmetric, then
\begin{equation}\label{5.qlog-1}
\|Q_h\psi\|\le C h^{3/2}
\|\psi\|,\quad\forall\psi\in S_h.
\end{equation}
Hence, for the solution of \eqref{lumped} with $v_h=P_hv$, we have
\begin{equation}\label{lumped-nonsmooth-est-opt2}
\| \tu_h(t) -  u(t)\| \le Ch^{3/2}
t^{-1}\| v \|, \for t>0.
\end{equation}
\end{theorem}
\begin{proof} 
Following again the steps in the proof of
Theorem \ref{verify}, we note that  \eqref{4.I} still holds, and it remains
 to bound $II$. For each internal vertex
$\zeta_0$ of one of the $\T_h(\Om_k)$,  the corresponding patch
$\Pio0$ is a $O(h^2)$ perturbation of a  symmetric patch, and thus \eqref{4.almsym} holds.
For  $\zeta_0\in Z_h^0$ a vertex on the boundary of two of the
$\T_h(\Om_k)$ we see that by  \eqref{Mh-def}
\[
 |M_h \chi(\zeta_0)|\le Ch^3\max_{x\in\Pio0}|\nabla\chi(x)|\le
Ch^2\|\nabla\chi\|_{L_2(\Pio0)},
\]
 and by the use of approximation properties of
the interpolation operator $\Id_h$ we get
\begin{equation}\label{4.3.2}
|M_h \Id_h\vfy(\zeta_0)|\le Ch^2\| \Id_h \vfy\|_{H^1(\Pio0)}
\le Ch^2 \big ( \|\vfy\|_{H^1(\Pio0)} + h |\vfy |_{H^2(\Pio0)} \big ).
\end{equation}
Using \eqref{4.almsym} and \eqref{4.3.2}  as earlier \eqref{5.delst} in \eqref{eh-bound},
 we conclude
\begin{align*}
|[\w, M_h \Id_h\vfy]|
\le Ch^2\ell_h^{1/2}\|\psi\|\,|\vfy|_2
+Ch\|\psi\|\,\|\vfy\|_{H^1(\Om_S)},
\end{align*}
where $\Om_S$ is a strip of width $O(h)$ around the interface between
the 
subdomains $\Om_k$ of $\Om$. Using now the inequality
$\|\vfy\|_{H^1(\Om_S)}\le Ch^{1/2}\|\vfy\|_{H^2(\Om)}
\le Ch^{1/2}\|\chi\|$,
we get
\begin{equation}\label{4.21}
|[\w, M_h \Id_h\vfy]|
\le Ch^{3/2}\|\psi\|\,\|\chi\|, \quad\forall\psi,\chi\in S_h,
\end{equation}
and hence
$|II|\le Ch^{3/2}\|\psi\|$. Together with \eqref{4.I}, this completes the proof of \eqref{5.qlog-1}.
The error estimate \eqref{lumped-nonsmooth-est-opt2} now follows by
Lemma \ref{lumped-L2-norm-nonsmooth} and
\[
\|\wtE_h(t)\wtDe_hQ_hP_hv\|
\le Ct^{-1}\|Q_hP_hv\|
\le Ch^{3/2}t^{-1}\|v\|,\for t>0.
\qedhere
\]
\end{proof}

We remark that the operator $M_h$ used here, modulo a constant
factor, is the same as the operator $\De_h^*$ in \cite{clt11}. The arguments
in the proofs of Theorems \ref{verify1} and \ref{th4.3} therefore show that the following
result holds for the lumped mass method.

\begin{corollary} \label{4.1}
Assume that $\{\T_h\}$ is almost or piecewise almost symmetric.
Then the nonsmooth data error estimates  for the lumped mass method,
corresponding to \eqref{5.qlog} 
and \eqref{5.qlog-1}, 
respectively, hold.
\end{corollary}

We finish this section by remarking that, in one space dimension,
the full $O(h^2)$ $L_2$ norm bound \eqref{higher_error}
 for $Q_h$ holds also
for almost symmetric partitions, without a logarithmic factor.
Let  $\Omega =(0,1)$ be partitioned by $0=x_0 < x_1 <
\dots < x_{N_h+1}=1$. Denote now $\T_h=\{\tau_i\}_{i=1}^{N_h+1}$, with
$\tau_i=[x_{i-1},x_i]$, and let $S_h$ be the set of
the continuous piecewise linear functions over  $\T_h$, vanishing 
at $x=0,1$. We set $h_i=x_i-x_{i-1}$ and $h=\max_i h_i$. The control
volumes are $V_i= (x_i-h_i/2,x_i+h_{i+1}/2)$ and
$J_h\psi(x)=\psi(x_i)$ for $x \in V_i$. We say that $\T_h$ is almost
symmetric if $|h_{i+1}-h_i|\le Ch^2$ for all $i$.

Simple calculations show, with $(\chi,\psi)=\int_0^1\chi\psi\,dx$ and
$\lla\chi,\psi\rra=(\chi,J_h\psi)$, for $\chi,\psi\in S_h$,
\[
\vep_h(\psi,\chi)=
\lla\psi,\chi\rra-(\psi,\chi)
=-  \dfrac1{24}
\sum_{i=1}^{N_h}\psi_i
\big (h_{i+1}(\chi_{i+1}-\chi_{i})-
h_i(\chi_{i}-\chi_{i-1}) \big ),
\]
where $w_i=w(x_i)$ for a function $w$ on $\Om$,
and the one--dimensional version of \eqref{Mh-def} at
$x_i$  becomes
\[
M_h\Id_h\vfy(x_i)= -\dfrac1{24}\big(h_{i+1}(\vfy_{i+1}-\vfy_{i})+
h_i(\vfy_{i-1}-\vfy_{i})\big),\quad i=1,\dots, N_h.
\]
The crucial step to prove \eqref{higher_error} is then to show an
analogue of \eqref{5.delst}, in this case
\begin{equation}\label{5.est}
| M_h \Id_h\vfy(x_i) | \le C h^{5/2}
\|\vfy\|_{H^2(\Pioi)}, \quad i=1,\dots,N_h,
 \with \Pioi=\tau_i\cup \tau_{i+1}, 
\end{equation}
from which \eqref{higher_error} follows as earlier.
Using the Taylor formula
\[
\vfy(x)=\vfy(x_i)+(x-x_i)\vfy'(x_i)+\int_{x_i}^x(x-y)\,\vfy''(y)\,dy,
\]
we find easily
\[
M_h\Id_h\vfy(x_i)=-\frac1{24}
(h_{i+1}^2-h_i^2)\vfy'(x_i)+O\big(h^{5/2}\,\|\vfy''\|_{L_2(\Pioi)}\big),
\quad i=1,\dots, N_h.
\]
By the almost symmetry,
$|h_{i+1}^2-h_i^2|\le C h^3$ and by the
Sobolev type inequality
\[
|\vfy'(x_i)|\le Ch^{-1/2}
\big(\|\vfy'\|_{L_2(\Pioi)}
+h\|\vfy''\|_{L_2(\Pioi)}\big)\le Ch^{-1/2}\|\vfy\|_{H^2(\Pioi)},
\]
for $i=1, \dots, N_h$, we now conclude that \eqref{5.est} holds.


\section{Examples of nonoptimal nonsmooth initial data estimates}\label{sec:counterexample}


In this section we present two examples where the necessary and sufficiency
 condition \eqref{opt-assumption} for an optimal $O(h^2)$ 
nonsmooth data error estimate for $t>0$ is not
satisfied. In the first example we construct a family of nonsymmetric meshes
$\{\T_h \}$ for which the norm on the left hand side of
\eqref{opt-assumption}
is bounded below by $ch$,
thus showing that the first order error bound of
Theorem \ref{t4.3} is the best possible. In the second example we
exhibit a piecewise symmetric mesh for which this norm
is bounded below by $ch^{3/2}$, implying that the error estimate of 
Theorem \ref{th4.3} is best possible.

\begin{figure}[t]
\centering{
\includegraphics*[height=0.23\textheight,width=0.6\textwidth]
{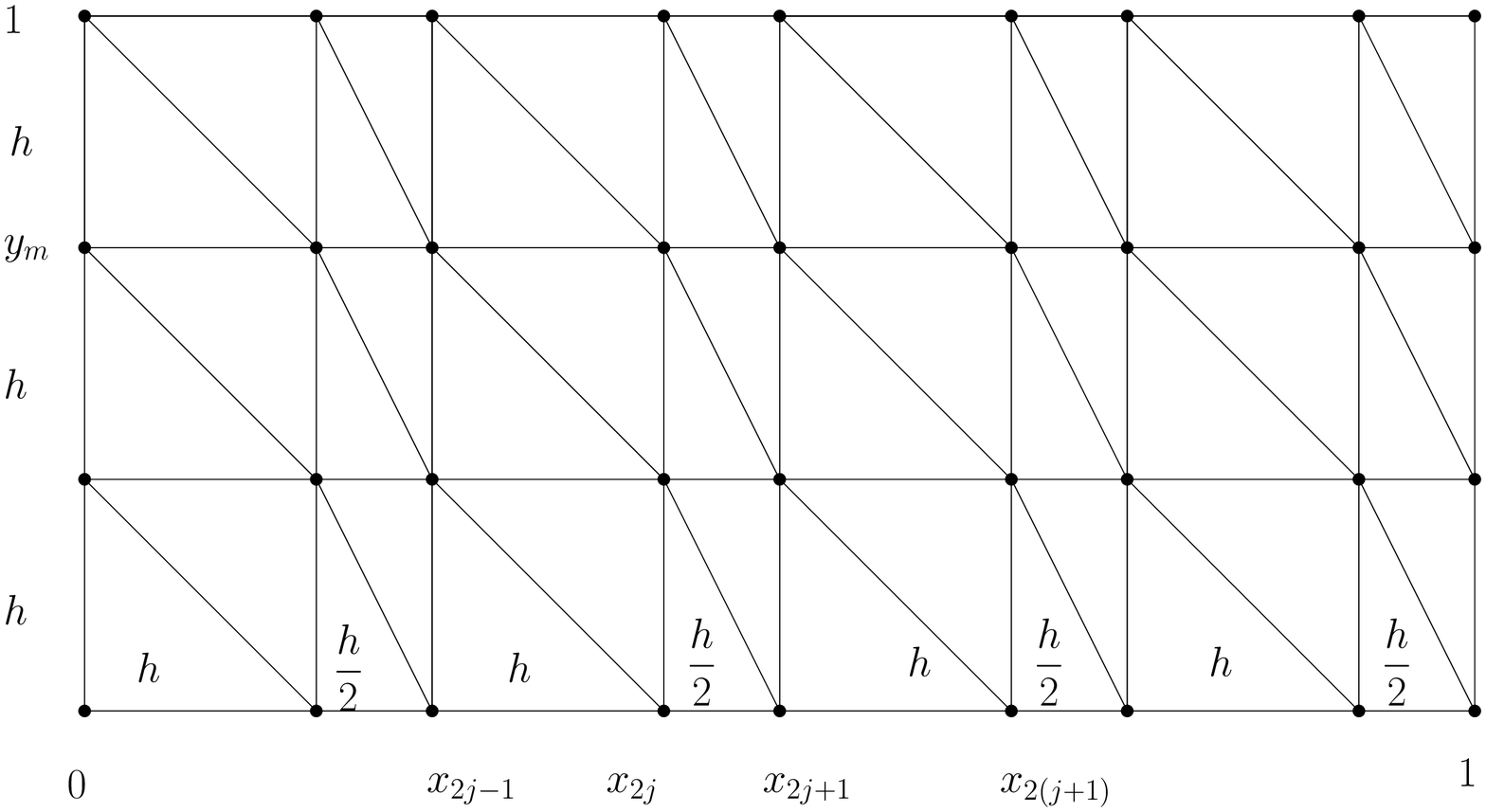}\hspace{5pt}
\includegraphics*[width=0.3\textwidth]{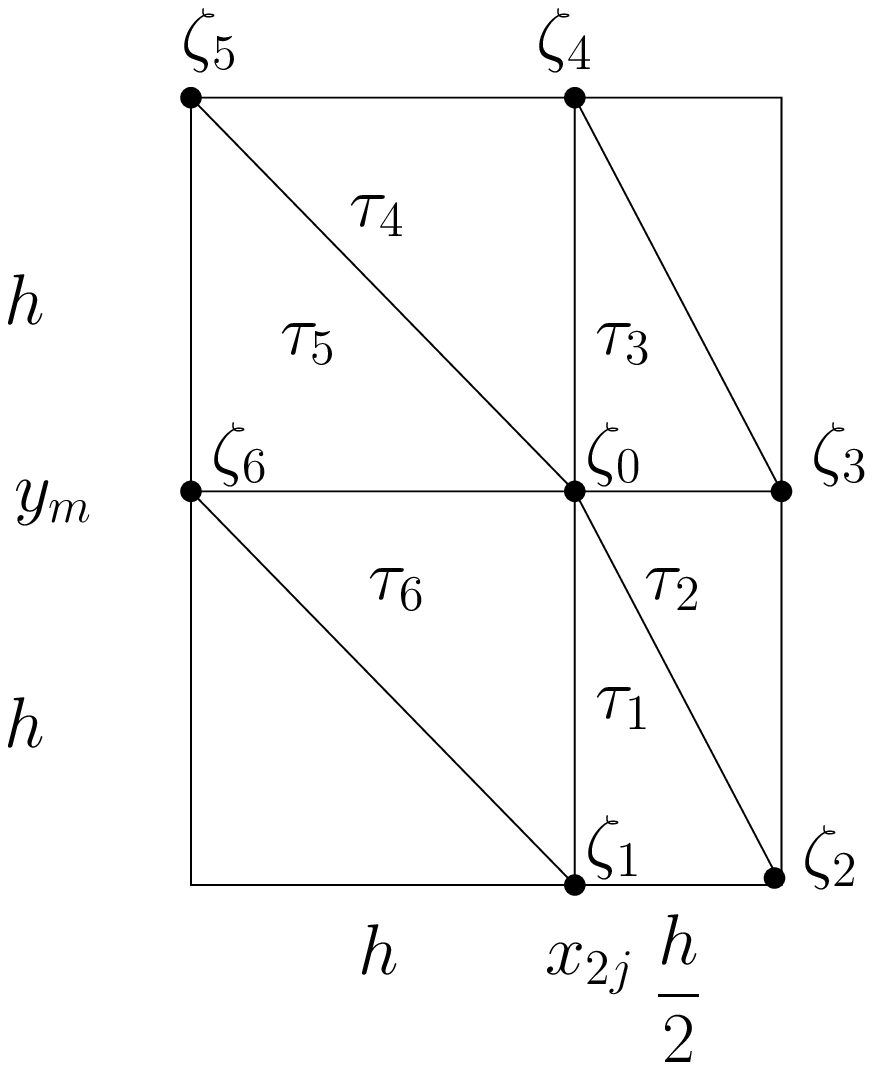}
} \caption{{\it Left}: A nonsymmetric mesh.
{\it Right}: A nonsymmetric patch $\Pio0$, around $\z_0$. } \label{fig:nonsym-mesh3}
\end{figure}

In our first example we choose $\Om=(0,1)\times(0,1)$ and introduce a 
quasiuniform family of triangulations $\{\T_h\}$ of $\Om$ as follows. Let $N$ be a 
positive integer divisible by $4$,  $h=4/(3N)$, $x_0=0$, and set, for 
$j=1,\dots,N$ and
$m=0,1,\dots,M=\frac34N$,
\begin{equation}\label{mesh-3}\begin{split}
&x_j=x_{j-1}+\begin{cases} \tfrac12 h,\ &\for j \text{ odd},\\
 h,\ &\for j \text{ even},
 \end{cases}\ \text{ and }\quad y_m=mh.
\end{split}\end{equation}
We split the rectangle
$(x_j,x_{j+1})\times (y_{m},y_{m+1})$ into two triangles 
by connecting the nodes $(x_j,y_m)$ and $(x_{j+1},y_{m-1})$, see 
\mbox{Figure~\ref{fig:nonsym-mesh3}.} This defines a triangulation $\T_h$ that is
not symmetric at any vertex.

Let now  $\z_0=(x_{2j},y_m)$, $\z_0\in Z_h^0$, and let $\Pio0$ be the 
corresponding nonsymmetric patch shown in Figure \ref{fig:nonsym-mesh3},
with vertices $\{\z_j\}_{j=1}^6$.  Let $\tau_j$ be the triangle 
in $\Pio0$ with vertices $\z_0$, $\z_j$, $\z_{j+1}$, where $\z_7=\z_1$.
We then  have $|\tau_j|=\tfrac14 h^2$, for $j=1,2,3$, and
$|\tau_j|=\tfrac12 h^2$, for $j=4,5,6$.
Thus, using   \eqref{Mh-def}, for $\psi\in S_h$, we obtain with $\psi_j=\psi(\z_j)$,
\begin{equation}\label{Mh_psi}
\begin{split}
M_h\psi(\z_0)=&-\frac1{54}\sum_{j=1}^6\om_j(\psi_j-\psi_0)
=-\frac1{54}\dfrac{h^2}4\Big(3(\psi_1+\psi_4-2\psi_0)
\\
&+2(\psi_2-\psi_0)+2(\psi_3-\psi_0)+4(\psi_5-\psi_0)
+4(\psi_6-\psi_0)\Big).
\end{split}
\end{equation}
Because $\nabla\psi$ is piecewise constant over $\Pio0$, 
we easily see that \eqref{Mh_psi} implies
\begin{equation}\label{Mh_bound}
|M_h\psi(\z_0)|
\le
Ch^2\|\nabla \psi\|_{L_2(\Pio0)},
\quad\forall \psi\in S_h.
\end{equation}

For a smooth function $\vfy$ we have, by  Taylor expansion, 
\[\vfy(\z_j)-\vfy(\z_0)=\nabla\vfy(\z_0)\cdot(\z_j-\z_0)+O(h^2),
\]
where $\z_j$ is considered as a vector with components its Cartesian coordinates
and the dot denotes the Euclidean inner product in $\R^2$.
Employing  this in \eqref{Mh_psi}, we find, after a simple calculation,
\begin{equation}\label{Mh_psi-1}
 M_h\Id_h\vfy(\z_0)=\dfrac{h^3}{108}\nabla\vfy(\z_0)\cdot(3,-1)
+O(h^4).
\end{equation}
 
Let $\fy_1(x,y)=2\sin(\pi x)\sin(\pi y)$ be the eigenfunction of $-\De$,  
corresponding to the smallest eigenvalue $\la_1=2\pi^2$. We then easily find
that
 $\nabla\fy_1(1/4,1/4)\cdot(3,-1)=2\pi$. Hence, there exists a square 
$\mathcal P=[1/4-d,1/4+d]^2$, with $0<d< 1/4$,  such that 
\begin{equation}\label{fy1-bound}
 \nabla\fy_1(z)\cdot(3,-1)\ge 1,\quad\forall z\in\mathcal P.
\end{equation}
Letting now for $z\in Z_h^0\cap \mathcal P$  we then have that
$M_h\Id_h\fy_1(z)\ge ch^3$, $c>0$, for $h$ small. 
 We shall prove the following proposition.
\begin{proposition}
\label{7.prop2} Let $\T_h$ be defined by \eqref{mesh-3}, 
$\mathcal P_h=\big\{z=(x_{2j},y_m)\in\mathcal P\big\}$ and consider
the initial value problem \eqref{lumped} with 
$v_h=\sum_{z\in\mathcal P_h}\Phi_{z}$,
where $\Phi_{z}\in S_h$ is the nodal basis function of $S_h$ at $z$.
Then we have, for $h$ small,
\[
\|\wtE_h(t)\wtDe_h Q_hv_h\| \ge c(t) h\|v_h\|, \with c(t)>0,\for
t>0.
\]
\end{proposition}
\begin{proof}
Letting $\wt\la^h_j$ and $\tpj$ be the eigenvalues and eigenfunctions of 
$-\wt\De_h$, and using Parseval's relation in $S_h$, equipped with
$\lla\cdot,\cdot\rra$, we have
\begin{equation}\label{Pro5.5}
\tribar\wtE_h(t) \wtDe_hQ_hv_h\tribar^2=\sum_{j=1}^{N_h} e^{-2t\tlj}
\lla\wt\De_hQ_hv_h,\tpj\rra^2
\ge e^{-2t\tilde\la_1^h}
\lla\wt\De_hQ_hv_h,\wt\phi^h_1\rra^2.
\end{equation}
 Combining 
\eqref{bDe_h-def},
 \eqref{q_h-def} and \eqref{Mh-def0}, we find
\begin{equation}\label{eq5.4}
-\lla\wtDe_h Q_h v_h,\psi\rra=(\nabla Q_hv_h,\nabla\psi)
=\vep_h(v_h,\psi)
=[ v_h,M_h\psi],\quad\forall \psi\in S_h.
\end{equation}
Note now that  for $z\in\mathcal P_h$, the corresponding patch $\Pi_z$ has the 
same 
form as the patch $\Pio0$ considered above. Thus employing \eqref{Mh_bound} for 
$\z_0=z$ we get, for $\psi\in S_h$,
\begin{equation}\label{Mh_bound-1}
|[ v_h, M_h\psi]|\le \sum_{z\in\mathcal P_h}|[\Phi_{z},M_h\psi]|
= \sum_{z\in\mathcal P_h}|M_h\psi(z)|\\
\le
Ch\|\nabla\psi\|,
\end{equation}
where in the last inequality we have used the fact that 
the number of points in $\mathcal P_h$ is $O(N^2)=O(h^{-2})$.
We recall from 
\cite{Liang_MZ} that
\begin{equation*}
\|\tilde\phi^h_1 -\fy_1 \|_{H^1} = O(h) \quad \mbox{ and } \quad
\tilde\la_1^h\to \la_1,
 \mbox{ as }  h\to 0, 
\end{equation*}
and, since obviously
$\|\fy_1-\Id_h\fy_1\|_{H^1}=O(h)$, 
 \eqref{Mh_bound-1}  with
$\psi=\tilde\phi_1^h-\Id_h\phi_1$, gives
\begin{equation}\label{Mh_bound-2}
|[ v_h,M_h(\tilde\phi_1^h-\Id_h\phi_1)]|
\le Ch\|\nabla(\tilde\phi_1^h-\Id_h\phi_1)\| \le Ch^2.
\end{equation}
For every $z\in \mathcal P_h$, \eqref{fy1-bound} holds, and thus, using
 \eqref{Mh_psi-1} with $\vfy=\fy_1$ and $\z_0=z$, we obtain,  
 for $h$ small, since  the number of vertices in 
$\mathcal P_h$ is bounded below by $cN^2$,
\[
[ v_h,M_h\Id_h\fy_1]
=\sum_{z\in\mathcal P_h}M_h\Id_h\fy_1(z)\ge ch^3N^2=ch, \with c>0.
\]
Combining  this with \eqref{Mh_bound-2}, we obtain, for $h$ small,
\[
[ v_h,M_h\tilde \fy_1^h]\ge [ v_h,M_h \Id_h\fy_1]
-|[ v_h,M_h(\tilde \fy_1^h-\Id_h\fy_1)]|\ge ch- Ch^2\ge ch, \with c>0.
\]
Since $\tribar v_h\tribar=O(1)$,   \eqref{Pro5.5} and \eqref{eq5.4} now show
\[
\tribar\wtE_h(t) \wtDe_hQ_hv_h\tribar\ge e^{-t\tilde\la_1^h}
[ v_h,M_h\tilde \fy_1^h]\ge
c(t)h\,\tribar v_h\tribar, \for t>0.
\]
Since $\tribar\cdot\tribar$ and $\|\cdot\|$ are 
equivalent norms, the proof is complete.
\end{proof}

It follows from Proposition \ref{7.prop2} and Lemma \ref{lumped-L2-norm-nonsmooth}  that the 
highest order of
convergence that can hold,
 uniformly for all $v\in L_2$, 
 and for any  family of triangulations $\{\T_h\}$, is
$O(h)$, i.e.,  Theorem \ref{t4.3old} is  best possible, in this case.

\begin{figure}[t]
\centering{
\includegraphics*[width=0.6\textwidth]{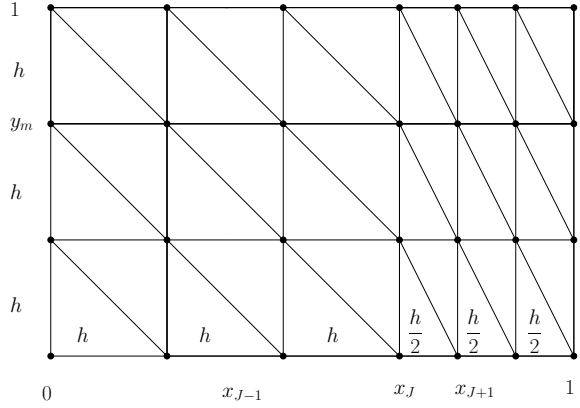}
} \caption{A piecewise symmetric mesh.} \label{fig:nonsym-mesh}
\end{figure}

We now turn to our second example, in which
 $\{\T_h\}$ is a piecewise  symmetric family.
Let again $\Om=(0,1)\times(0,1)$ and consider a 
triangulation $\T_h$ of $\Om$, 
where the nodes $(x_j,y_m)$ are given 
as follows. With $J$ a positive integer, 
 let $N=7J$, $M=4J$ and $h=1/(4J)$, and set for $j=0,\dots,N$ and $m=0,\dots,M$,
\begin{equation}\label{mesh-1}
x_j=\begin{cases}
jh, &\for 0\le j\le J,\\
1/4+(j-J)h/2, &\for J<j\le N,
\end{cases}
\quad \andy \quad 
 y_m=mh,
\end{equation}
 see Figure \ref{fig:nonsym-mesh}.
This time we consider the set of vertices in $\mathcal P$ with $x=1/4$ and 
prove the following proposition.
\begin{proposition}
\label{7.prop} Let $\T_h$ be defined by \eqref{mesh-1} and $\mathcal
P_h^\prime=\big\{z=(x_J,y_m) \in\mathcal P\big\}$. For the initial value problem
\eqref{lumped}, with $v_h=\sum_{z\in\mathcal P_h^\prime}\Phi_{z}$, where
$\Phi_{z}\in S_h$ is the nodal basis function of $S_h$ at  $z$, we
have, for $h$ small,
\[
\|\wtE_h(t)\wtDe_h Q_hv_h\| \ge c(t) h^{3/2}\|v_h\|, \with
c(t)>0,\for t>0.
\]
\end{proposition}
\begin{proof}
Again, using 
\eqref{Pro5.5} and \eqref{eq5.4},  we have,
\begin{equation}\label{5a}
\tribar\wtE_h(t) \wtDe_hQ_hv_h\tribar^2\ge e^{-2t\tilde\la_1^h}
[v_h,M_h\tilde\fy_1^h]^2.
\end{equation}
For $z\in\mathcal P_h^\prime$, the corresponding patch $\Pi_z$  has the 
same form as the patch $\Pio0$ considered above, see Figure \ref{fig:nonsym-mesh3} (right). 
Thus employing \eqref{Mh_bound} for 
$\z_0=z$ and taking into account that the number of vertices in $\mathcal P_h'$ is $O(N)$
we now obtain,  for $\psi\in S_h$,

\begin{equation*}
|[ v_h, M_h\psi]|\le \sum_{z\in\mathcal
P_h^\prime}|[\Phi_{z},M_h\psi]|
=\sum_{z\in\mathcal P_h^\prime}|M_h\psi(z)|
\le Ch^{3/2}\|\nabla\psi\|.
\end{equation*}
Similarly to  \eqref{Mh_bound-2} this now shows
\begin{equation}\label{Mh_bound-3}
|[ v_h,
M_h(\tilde\fy_1^h-\Id_h\fy_1)]|\le Ch^{5/2},
\end{equation}
and, again using \eqref{Mh_psi-1}, for $h$ small,
\[[ v_h,M_h\Id_h\fy_1]
=\sum_{z\in\mathcal P_h^\prime}M_h\Id_h\fy_1(z)\ge ch^3J=ch^2, \with
c>0.
\]
Combined with \eqref{Mh_bound-3} this gives, for $h$  small,
\begin{equation}\label{5.15}
[ v_h,M_h\tilde \fy_1^h]\ge ch^2- Ch^{5/2}\ge ch^2, \with c>0.
\end{equation}
Since $\tribar v_h\tribar=O(h^{1/2})$ we obtain from \eqref{5a} and \eqref{5.15}
\[
\tribar\wtE_h(t) \wtDe_hQ_hv_h\tribar\ge
c(t)h^2\ge c(t)h^{3/2}\,\tribar v_h\tribar, \for t>0.\qedhere
\]
\end{proof}
It follows from Proposition \ref{7.prop} and Lemma \ref{lumped-L2-norm-nonsmooth} that the 
highest order of
convergence that can hold,
 uniformly for all $v\in L_2$, 
 and for all piecewise symmetric families $\{\T_h\}$,  is
$O(h^{3/2})$, i.e.,  Theorem \ref{th4.3} is  best possible in this regard.

\begin{remark}
Since $M_h$ is proportional to the operator
$\De_h^*$ used in \cite{clt11}, the arguments in this section
also apply to the lumped mass method. In particular,
the analogue of Proposition \ref{7.prop2} then shows that the first order nonsmooth data
estimate for $t>0$ of \cite[Theorem 4.3]{clt11} is best possible for
general
triangulations $\{T_h\}$. Further, the $O(h^{3/2})$ estimate stated in Corollary \ref{4.1} is
best possible for piecewise almost symmetric triangulations. 
Our examples here may be thought of as generalizations to two space
dimensions of the one--dimensional counter--examples in \cite[Section
7]{clt11}.
\end{remark}


\section{Some fully discrete schemes}\label{fully}


In this  section we discuss briefly the generalization of our above 
results for the spatially semidiscrete finite volume method to
some basic fully discrete schemes, namely the
 backward Euler and Crank-Nicolson methods.

With $k>0$, $t_n=n\,k,\ n=0,1,\dots$, the backward Euler finite
volume method approximates $u(t_n)$ by
 $\tU^n\in S_h$ for $n\ge0$ such that,
with  $\bar\partial \tU^n=(\tU^n-\tU^{n-1})/k$,
\[
\lla\bar\partial \tU^n,\chi\rra+(\nabla
\tU^n,\nabla\chi)=0,\quad\forall\chi\in S_h, \for n\ge1,\with \tU^0=v_h,
\]
or, 
\begin{equation}\label{lmvvv}
\bar\partial \tU^n -\wtDe_h\tU^n=0, \for n\ge1,\with \tU^0=v_h.
\end{equation}


Introducing the discrete solution operator $\wtE_{kh}=(I-k\wtDe_h)^{-1}$ we may write
$\tU^n=\wtE_{kh}\tU^{n-1}=\wtE_{kh}^n\tU^0$, $n\ge1$. Using eigenfunction expansion and 
 Parseval's relation, we obtain,  analogously to \cite[Chapter 7]{Thomee06}, the stability property
\begin{equation}\label{Ekh-stability}
 \|\nabla^p\wtE_{kh}^n\chi\|\le C\|\nabla^p\chi\|,\quad\forall\chi\in S_h,\for p=0,1.
\end{equation} 

The estimates that follow and their proofs are analogous to those for the lumped mass
method derived in \cite{clt11}, since the operators $\wtE_h(t)$,
$\wtDe_h$ and $Q_h$, defined in Section \ref{sec:prelim}, have
properties analogous to those of the corresponding operators for the lumped mass 
method. For simplicity we will only sketch the proof of Theorem \ref{lmfully}.

We shall use the following abstract lemma shown in \cite{clt11}, in
the case $\H=S_h$, normed by $\tribar\cdot\tribar$, and  with
$\bA=-\wtDe_h$.
\begin{lemma}
\label{fully-BE} Let $\bA$ be a linear, selfadjoint, positive definite
operator in a Hilbert space $\H$, with compact inverse, let $\bfu=\bfu(t)$
be the solution of
\begin{equation*}
 \bfu'+\bA\bfu=0,\for t>0,\with \bfu(0)=\bfv,
\end{equation*}
and let $\bfU=\{\bfU^n\}_{n=0}^\infty$ be defined by
\begin{equation*}
\bar\partial \bfU^n+\bA\bfU^n=0,\for n\ge1,\with \bfU^0=\bfv.
\end{equation*}
Then, for $p=0,1$, $-1 \le q\le 3,$ with $p+q\ge0$, we have 
\begin{equation*}
\|\bA^{p/2}(\bfU^n-\bfu(t_n))\|\le
 Ckt_n^{-(1-q/2)}\|\bA^{(p+q)/2}\bfv\|,\for n\ge1.
\end{equation*}
\end{lemma}

The error estimates of the following theorem for \eqref{lmvvv}
 are of optimal order  under the same assumptions as in Section
\ref{sec:smooth}.

\begin{theorem}\label{lmfully} Let
 $u$  and $\tU$ be the solutions of \eqref{eq1}
and  \eqref{lmvvv}.
 Then, for $n\ge1$,
\begin{equation*}
\| \tU^n-u(t_n)\| \le 
\begin{cases}
C(h^2+k) |v|_2,&\ifff \ \|v_h-v\|\le Ch^2|v|_2,\\
C(h^2+k) t_n^{-1/2}|v|_1,&\ifff \ v_h=P_hv  \andy  \|\nabla P_hv\|\le C|v|_1,\\
C(h^2+k) t_n^{-1}\|v\|,&\ifff\  v_h=P_hv \andy \eqref{higher_error}\ holds.
\end{cases}
\end{equation*}
\end{theorem}
\begin{proof} 
Analogously  to the proof of \cite[Theorem 8.1]{clt11}, 
we split the error as
\begin{equation*}
 \tU^n-u(t_n)=(\tU^n-\tu_h(t_n))+(\tu_h(t_n)-u(t_n))=\be_n+\eta_n.
\end{equation*}
By Theorems \ref{lumped-L2-norm-smooth}, \ref{lumped-L2-norm-msmooth}
and \ref{lumped-nonsmooth-opt}, $\eta_n$ is bounded as required. In
order to bound $\be_n=(\wtE_{kh}^n-\wtE_h(t_n))v_h$ in the smooth data
case, it suffices, using the stability estimates \eqref{Ekh-stability}
and Lemma \ref{bEh-estimates},  to consider $v_h=R_hv$. We  obtain by
Lemma \ref{fully-BE}, with $\bA=A_h=-\wtDe_h$, and $q=2,1,0$,
\begin{equation*}
\tribar\be_n\tribar=
\tribar \tU^n-\tu_h(t_n)\tribar
\le Ckt_n^{-(1-q/2)}\tribar A_h^{q/2} v_h\tribar \le
Ckt_n^{-(1-q/2)}|v|_q,
\end{equation*}
where for $q=2$, the last inequality follows from 
\[
\tribar A_h R_hv\tribar^2=(\nabla R_hv,\nabla A_h R_hv) 
=(\nabla v,\nabla A_hR_hv)=-(\De v,A_h R_hv),
\]
 for $q=1$ from $\tribar A_h^{1/2}P_hv\tribar=\|\nabla P_hv\|\le
C|v|_1$ and for $q=0$ from $\tribar P_hv \tribar\le C\|v\|$.
\end{proof}

Also for the lumped mass method the analogous result in the
mildly nonsmooth data case $v\in \Hdot^1$ holds,  and should replace
the result for $q=1$ in \cite[Theorem 8.1]{clt11}, cf. the remark
after Theorem \ref{lumped-L2-norm-msmooth}.
 
Recall that $Q_h$ satisfies \eqref{higher_error} if $\{\T_h\}$ is
symmetric. For almost symmetric or piecewise almost symmetric
$\{\T_h\}$ we obtain  correspondingly the following nonsmooth initial
data error estimates employing \eqref{lumped-nonsmooth-est-opt1} and 
\eqref{lumped-nonsmooth-est-opt2}.

\begin{theorem}\label{lmfully-almost}
 Let $u$ and $\tU$ be the solutions of \eqref{eq1} and  \eqref{lmvvv},
with $v_h=P_hv$. Then, for $n\ge1$,
\begin{equation*}
\| \tU^n-u(t_n)\| \le 
\begin{cases}
C(h^2\ell_h^{1/2}+k)t_n^{-1} \|v\|,&\ifff \{\T_h\} \text{ is almost symmetric} ,
\\
C(h^{3/2}+k) t_n^{-1}\|v\|,&\ifff \{\T_h\} \text{ is piecewise almost symmetric}.
\end{cases}
\end{equation*}
\end{theorem}

For the gradient of the error we may prove as in \cite[Theorem
8.2]{clt11},  the following smooth and
nonsmooth data error estimates, without additional assumptions on
 $\T_h$.  For smooth initial data we assumed in \cite{clt11} that
$v_h=R_hv$, but  the more general choices of $v_h$ are permitted
by the stability estimates \eqref{Ekh-stability} and Lemma
\ref{bEh-estimates}.

\begin{theorem}\label{lmfully-H1} Let
 $u$ and $\tU$ be the solutions of \eqref{eq1} and  \eqref{lmvvv}.
 Then, for $n\ge1$,
\begin{equation*}\label{lmfully-est2}
\|\nabla( \tU^n-u(t_n))\| \le
\begin{cases}
C(h+k)|v|_3,&\ifff \ \|\nabla(v_h-v)\|\le Ch|v|_2,
\\
C(h\,t_n^{-1}+k\,t_n^{-3/2})\|v\|,&\ifff \ v_h=P_hv.
\end{cases}
\end{equation*}
\end{theorem}

We now turn to the Crank--Nicolson  method,
 defined by
\begin{equation}\label{cnvvv}
\bar\partial \tU^n-\wtDe_h \tU^{n-\tfrac12}=0, \for n\ge1, 
\with U^0=v_h, \ \  \tU^{n-\tfrac12}=\tfrac12(\tU^n+\tU^{n-1}).
\end{equation}
Denoting again the discrete solution operator by 
$\wtE_{kh}=(I+\frac 12k\wtDe_h)(I-\frac 12k\wtDe_h)^{-1}$ 
we may write
$\tU^n=\wtE_{kh}\tU^{n-1}=\wtE_{kh}^n\tU^0$, $n\ge1$. Using eigenfunction expansion and 
 Parseval's relation,  we find   
that \eqref{Ekh-stability}  also holds for this 
method.

The Crank--Nicolson method does not have  as advantageous smoothing
properties as the backward Euler method, which is reflected in the
fact
that the following analogue of Lemma \ref{fully-BE}, shown in
\cite[Lemma 8.2]{clt11}, does not allow $q=0$.
\begin{lemma}\label{fully-CN}
Let  $\bA$ and $\bfu(t)$ be as in Lemma \ref{fully-BE} and let
$\bfU^n$
satisfy
\[
\bar\partial \bfU^n+\bA \bfU^{n-\tfrac12}=0,\for n\ge1,\with
\bfU^0=\bfv.
\]
Then
\begin{equation*}\label{fully-est2}
\|\bA^{p/2}(\bfU^n-\bfu(t_n))\|\le 
Ck^2t_n^{-(2-q)}\|\bA^{p/2+q}\bfv\|, \for
n\ge1,\ p=0,1,\ q=1,2.
\end{equation*}
\end{lemma}

This time optimal order estimates for the error in $L_2$ and in $H^1$, hold uniformly
down to $t=0$, if $v\in \dot H^4$ and $v\in \dot H^5$, respectively.
The proofs  are analogous to those of  \cite[Theorems 8.3 and
8.4]{clt11},  where we assumed $v_h=R_hv$. Again the 
stability estimates \eqref{Ekh-stability} and Lemma
\ref{bEh-estimates} permit  the more general choices for $v_h$.
\begin{theorem}\label{lmfully-2}
Let $u$ and $\tU$   be the solutions of \eqref{eq1}
and \eqref{cnvvv}.  Then, with $q=1,2$, we have, for $n\ge1$,
\begin{align*}\label{lmfully-est3}
\| \tU^n-u(t_n)\| &\le C(h^2+k^2t_n^{-(2-q)})|v|_{2q},\ifff \ \|v_h-v\|\le Ch^2|v|_2\\
\| \nabla(\tU^n-u(t_n))\| &\le C(h+k^2t_n^{-(2-q)})|v|_{2q+1},
\ifff \ \|\nabla(v_h-v)\|\le Ch|v|_2.
\end{align*}
\end{theorem}

For optimal order convergence for initial data only
in $L_2$, one may modify the Crank--Nicolson scheme by taking the
first two  steps by the backward Euler method, which has a smoothing
effect. We may show then the following result,  analogously to that of
\cite[Theorem 8.5]{clt11}, with the obvious modifications for almost
symmetric and piecewise almost symmetric families $\{\T_h\}$. 

\begin{theorem}\label{lmfully-3} Let $u$  be the solution of
\eqref{eq1} and  $\tU^n$ that of
 \eqref{lmvvv}, for $n=1,2$, and of \eqref{cnvvv}, for
$n\ge3$, with $v_h=P_hv$ and assume \eqref{higher_error} holds. Then
we have
\begin{equation*}
\| \tU^n-u(t_n)\| \le
C(h^2t_n^{-1}+k^2t_n^{-2})\|v\|, \for n\ge1.
\end{equation*}
\end{theorem}

\section{Problems with More General Elliptic Operators}\label{sec:general}

This final section is devoted to the extension of our earlier results
to the more general problem \eqref{1.eq1-general}, and we recall that we shall
consider the finite volume method \eqref{1.fv-general-22} where the bilinear
form $\wt a_h(\cdot,\cdot)$ is defined by \eqref{1.a-h-modified}. Our
error analysis is again based on estimates for the standard Galerkin finite element
method, in this case defined by \eqref{1.fem-general} and \eqref{a:form}. It is well known
that for this method the stability and smoothing estimates \eqref{fem-reg}
hold as do the error estimates  \eqref{1.sm}--\eqref{1.hsm}, where
the norms $|\cdot|_q$ are defined analogously to the norms  
\eqref{norms}, using the eigenvalues and eigenfunctions of $A$.

We introduce the discrete elliptic operator $\Ah: S_h \to S_h$ by
\begin{equation}\label{A_h-def}
\lla \Ah \vv, \chi \rra = \tah(\vv,\I_h\chi), \quad \forall \chi,\vv \in S_h,
\end{equation}
which is symmetric and  positive definite with respect to the inner
product $\lla \cdot, \cdot \rra$ by \eqref{a-symmetry}, since
$(\qbar\psi,\I_h\chi)$ is symmetric, positive semidefinite on $S_h$.
This follows from the fact that $\int_\K\chi\I_h\psi\,dx$ is symmetric
by \eqref{4.5} and $\qbar$ is constant and nonnegative in each $\K$ of
$\T_h$. 
We may then rewrite \eqref{1.fv-general-22} as
\begin{equation}\label{A_h-problem}
 \tu_{h,t}+\Ah\tu_h=0, \for t\ge0, \with  \tu_h(0)=v_h,
\end{equation}
and the solution  is given by $\tu_h(t)=\wtE_h(t)v_h$, where 
$\wtE_h(t)=e^{-\Ah \,t}$ is defined as in \eqref{bEh}, with $\{\tlj\}$
and $\{\tpj\}$ the eigenvalues and  eigenfunctions of
$\Ah$, orthonormal with respect to $\langle \cdot, \cdot \rangle$.

Note that a slightly different finite volume element method for
\eqref{1.eq1-general} has been considered in \cite{msz}. This method
differs in the discretization of the lower order term, using the
 bilinear $\bar a_h(\cdot,\cdot)$  defined by
\[
 \bar
a_h(\psi,\I_h\chi)=(\alb\nabla\psi,\nabla\chi)+(\beta\I_h\psi,
\I_h\chi),\quad\forall\psi,\chi\in S_h.
\]
For this method analogous results to Theorems
\ref{gen-smooth}--\ref{gen-sym} hold.

Following our error analysis in the previous sections we
introduce $\de=\tu_h - u_h$ and split the error into
$\tu_h -  u= \de + (u_h - u)$, where 
 $ u_h - u$ and $\nabla (u_h - u)$ are estimated by the analogues of 
\eqref{1.sm}--\eqref{1.hsm}.
 It  therefore suffices to derive estimates for $\de$,
which satisfies, for $t\ge0$, 
\begin{equation}\label{8-error}
\lla \de_{h,t}, \chi \rra + \tah(\de, \I_h\chi)=-\vep_h(u_{h,t}, \chi)
 - \tveph(u_h,\chi),\quad \forall \chi\in S_h, \with \de(0)=0,
\end{equation}
where $\vep_h(\cdot,\cdot)$ is given by \eqref{veph-def} and 
$\tveph(\cdot,\cdot)$ is defined by
\begin{equation}\label{tveph-def}
\tveph(\vv,\chi)= \tah(\vv, \I_h\chi) - a(\vv,\chi), \quad\forall \vv, \chi \in S_h.
\end{equation}
Now let  $Q_h: S_h \to S_h$ and $\Qal: S_h \to S_h$ be the
quadrature error operators given  by
\begin{equation}
 \label{new-Qh}
\tah( Q_h \vv, \I_h\chi) =  \vep_h(\vv, \chi)
\ \andy\
\tah ( \Qal \vv,  \I_h\chi)  =  \tveph(\vv,
\chi),\quad\forall\vv,\chi\in S_h.
\end{equation}
Using \eqref{A_h-def}, the equation
\eqref{8-error} for $\de$ can then be written in operator form as
\begin{equation*}
\de_t + \Ah \de =  -\Ah Q_h u_{h,t} - \Ah \Qal u_h ,
 \for t\ge0, \with \de(0)=0.
\end{equation*}
This problem is similar to \eqref{delta-eq}, except that the
operator $-\wtDe_h$ is replaced by $\Ah$ and that on the right hand
side we have an additional term resulting from the approximation of
the bilinear form $a(\cdot, \cdot)$. By Duhamel's principle we 
have 
\begin{equation}\label{gen-Duhamel}
\begin{split}
\de(t) = & -\int_0^t \wtE_h(t-s) \Ah Q_h u_{h,t}(s)\, ds \\
        &  \qquad \quad - \int_0^t \wtE_h(t-s) \Ah \Qal u_{h}(s)\, ds  
=:\detil(t) + \dehat(t),\for t\ge0.
\end{split}
\end{equation}

To estimate $\de$ it therefore suffices to bound $ \detil $ and $
\dehat$. For this we need some auxiliary results, which are discussed
below.

\begin{lemma}\label{epal} Let $\all, \qc\in \C^2$. For the
error functional $\tveph$, defined by \eqref{tveph-def},  we have
\[
|\tveph(\vv,\chi)|\le Ch^{p+q}\|\nabla^q\vv\|\,\|\nabla^p\chi\|,\quad\forall\vv,
\chi\in S_h, \with p,q=0,1.
\]
\end{lemma}
\begin{proof} 
In view of \eqref{tveph-def}, we may write
\[
\tveph( \vv, \chi) = ((\alb -\all) \nabla \vv, \nabla \chi)
             +( \qbar \vv, \I_h\chi ) - (\qc \vv, \chi).
\]
We then split $ \tveph( \vv, \chi)$ as a sum of
integrals over $\K \in \T_h$. Since $\alb=\all(z_\K)$, we see
$\int_\K (f -f(z_\K)) dx=0 $ for  linear functions $f$,
and hence 
\begin{equation}\label{quad-error}
 |\int_\K (f -f(z_\K) dx|\le Ch^2_{\K}|\K|\|f\|_{\C^2},\for f\in
\C^2,
\end{equation}
with $h_\K$ the maximal side length of $\K$.
Therefore, using this and the fact that $\nabla  \vv \cdot \nabla
\chi$ is constant in $\K$, we get 
\begin{equation*} 
\big| \int_\K (\alb -\all) \nabla  \vv \cdot \nabla \chi\, dx \big|
 \le C h_\K^2 \|\all\|_{\C^2}  \int_\K \big |  \nabla  \vv \cdot \nabla \chi \big |\,dx
\le C h_\K^2 \| \nabla  \vv\|_{L_2(\K)} \|\nabla \chi \|_{L_2(\K)}.
\end{equation*}
Employing  an inverse inequality locally and 
summing over $\K\in\T_h$, we obtain
\begin{equation}\label{a_h-error}
|((\alb -\all) \nabla  \vv,
\nabla \chi)| 
 \le Ch^{p+q} \|\nabla^q   \vv \| \, \| \nabla^p \chi\|.
\end{equation}
In a similar manner we estimate the zero order term. Obviously,
\begin{equation}\label{q-error-split}
 (\qbar\vv,\I_h\chi)-(\qc\vv,\chi)=\vep_h(\qbar\vv,\chi)
+((\qbar-\qc) \vv,\chi).
\end{equation}
Using Lemma \ref{quad-error-lemma} we can bound the first term on the right--hand side
 of \eqref{q-error-split}, as desired. We then split the second
term, in the 
following way
\begin{equation}
\begin{split}\label{q-error}
 \int_\K (\qbar - \qc)   \vv \, \chi\, dx
 = & \int_\K (\qbar - \qc)   (\vv \chi)(z_\K)  dx   \\
 &+   \int_\K (\qbar - \qc)   (\vv \, \chi - (\vv \chi)(z_\K) )
dx =: I + II.
\end{split}
\end{equation}
Employing \eqref{quad-error}
 we easily get
\begin{align*}
 | I | \le Ch_\K^2 \|\qc \|_{\C^2}|\K|  | (\vv \chi)(z_\K) |
& =Ch_\K^2  |\K|^{-1} \big|\int_{\K} \vv\,dx\big|\  
\big|\int_{\K}\chi\,dx \big|\\
 &\le C h^2 \|  \vv\|_{L_2(\K)} \| \chi \|_{L_2(\K)},
\end{align*}
and since $|\qc-\qbar|\le Ch_{\K}\|\qbar\|_{\C^1}$ in $\K$,
\begin{align*}
 | II | &\le  Ch_\K^2 \int_\K (\big | \nabla \vv \,
\chi \big | +  \big | \vv \, \nabla \chi \big |)\,  dx\\
&\le Ch^2 (  \| \nabla  \vv\|_{L_2(\K)} \| \chi \|_{L_2(\K)} 
+\|  \vv\|_{L_2(\K)} \|\nabla \chi \|_{L_2(\K)}).
\end{align*}
Combining the bounds  for $I$ and $II$ with
\eqref{q-error}, using an inverse inequality locally, summing over
$\K\in\T_h$ and using \eqref{a_h-error}, we conclude the proof.
\end{proof}
For the solution operator $\wtE_h(t)=e^{-\Ah \,t}$ of
\eqref{A_h-problem}, one shows, 
as in Lemma \ref{bEh-estimates}, the following smoothing property.
\begin{lemma}\label{bEh-estimates-2lemma} For  $\wtE_h$, the solution operator of 
 \eqref{A_h-problem}, we have, for $v_h\in S_h$ and $ t>0$, 
\begin{equation*}\label{bEh-estimates-2}
\|\nabla^pD_t^\ell \wtE_{h}(t)v_h\| \le Ct^{-\ell-(p-q)/2} \|\nabla^q v_h\|,  
\,  \ell \ge 0, ~~ p,q=0,1, ~~ 2\ell+p\ge q.
\end{equation*}
\end{lemma}

Further, following the steps in the proof of Lemma \ref{q_h-stability} we can get 
easily the following estimate 

\begin{lemma}\label{q_a-stability-a} 
Let $\Ah$, $Q_h$ and $\Qal$ be the operators defined by
\eqref{A_h-def} and \eqref{new-Qh}.
 Then
\begin{equation*}
\|\nabla Q_h \chi \|+ h\|\Ah Q_h \chi \|\le C h^{p+1} \| \nabla^p
\chi \|, 
\quad \forall \chi \in S_h,\for p=0,1,
\end{equation*}
and the same bounds hold if we replace $Q_h$ by $\Qal$.
\end{lemma}
\begin{proof} 
Using the fact that $ \tah(\chi,\I_h\chi)\ge c \|\nabla \chi\|^2$, for $\chi\in S_h$,
  \eqref{new-Qh} and Lemma 
\ref{quad-error-lemma}, with $\psi=Q_h \chi$,  we obtain  for $p=0,1$,
\begin{equation*}
c \|\nabla  Q_h \chi \|^2  \le \tah(Q_h \chi, \I_h Q_h \chi)  = \vep_h(
\chi, Q_h \chi)\le Ch^{p+1} \|\nabla^p   \chi \| \, \| \nabla Q_h \chi
\|,
\end{equation*}
which bounds $Q_h\chi$ as desired.
By the definition of $\Ah$ and Lemma \ref{quad-error-lemma}  with
$q=0$, 
we also get for $p=0,1$,
\[
 \tribar\Ah Q_h \chi\tribar^2=\vep_h(\chi,\Ah Q_h\chi)\le 
  C h^{p} \| \nabla^p \chi \| \, \| \Ah Q_h\chi\|.
\]
Since the norms $\tribar\cdot\tribar$ and $\|\cdot\|$ are 
equivalent on $S_h$, this shows  the bound stated. 

To prove the corresponding bounds for $\Qal$, analogously we use Lemma
\ref{epal} instead of Lemma \ref{quad-error-lemma}.
\end{proof}

We now show an estimate for $\dehat$ defined in
\eqref{gen-Duhamel}, including exceptionally the exponential decay of
the bound.
\begin{lemma}\label{lem:de-2-est}
 For the error $\dehat$  defined by 
\eqref{gen-Duhamel}, we have
\begin{equation*}
\| \dehat(t)  \| + h\| \nabla \dehat(t) \| \le C h^2 e^{-ct}\| v_h
\|,\for t\ge0,\quad  v_h\in S_h, \with c>0.
\end{equation*}
\end{lemma}
\begin{proof}
Using the fact that $\wtE_h(t) \Ah = -D_t\wtE_h(t)$, Lemmas
\ref{bEh-estimates-2lemma} and
 \ref{q_a-stability-a}, and the smoothing property \eqref{fem-reg}, we
find this time taking into account the exponential decay of
$\wtE_h(t)$ and $u_h(t)$ for large $t$,
\[
 \begin{split}
\| \dehat(t)  \|  +  h\| \nabla \dehat(t) \| 
 &  \le \int_0^t \Big ( \| \wtE_h^\prime (t-s) \Qal u_{h}(s) \| 
+ h \| \nabla \wtE_h (t-s) \Ah \Qal u_{h}(s) \| \Big ) ds \\
& \le C \int_0^t (t-s)^{-1/2}e^{-c(t-s)} \Big ( \| \nabla \Qal
u_{h}(s) \| + 
h \| \Ah \Qal u_{h}(s) \| \Big ) ds \\
& \le C  h^2  \int_0^t (t-s)^{-1/2} e^{-c(t-s)}\| \nabla u_{h}(s) \|\,
 ds \\
& \le C h^2  \int_0^t (t-s)^{-1/2}e^{-c(t-s)} s^{-1/2}e^{-cs}\, 
ds\, \|v_h\| = C h^2 e^{-ct}\|v_h\|,
\end{split}
\]
which is the desired result.
\end{proof}

We are now ready for the error estimates for the solution of
\eqref{A_h-problem}.
\begin{theorem}\label{gen-smooth}
 Let $u$ and $\tu_h$ be the solutions of \eqref{1.eq1-general} and 
\eqref{A_h-problem}. Then for $t>0$,
\begin{equation*} 
\| \tu_h(t) - u(t) \|\le\begin{cases}C h^2 |v|_2, &\text{if } \|v_h-v\|\le Ch^2|v|_2,\\
C h^2t^{-1/2}|v|_1,&\text{if } v_h=P_hv \ \ 
\text{and}\ \ \|\nabla P_hv\|\le C|v|_1.
\end{cases}
\end{equation*}
Further, the estimates for the gradient of the error of Theorem
\ref{t4.3} remain valid.
\end{theorem}
\begin{proof}  As in Section \ref{sec:smooth},
 it is suffices to estimate $\de=\tu_h-u_h$. 
Using the splitting \eqref{gen-Duhamel}, 
$\de=\detil + \dehat $, the term $\dehat$ is easily  bounded by 
Lemma \ref{lem:de-2-est},  and $\detil$
is bounded as in  Theorems \ref{lumped-L2-norm-smooth} 
and \ref{lumped-L2-norm-msmooth}, now applying
 Lemmas \ref{bEh-estimates-2lemma} and \ref{q_a-stability-a}. 
\end{proof}

Turning to nonsmooth initial data, we begin with the following lemma.
\begin{lemma}
\label{lumped-L2-norm-nonsmooth-1} Let $u$ and  $\tu_h$  be the solutions of
\eqref{1.eq1-general} and 
\eqref{A_h-problem}.  Then for $t>0$
\begin{equation*}
\| \tu_h(t) -  u(t)-\wtE_h(t)\Ah Q_h v_h\| \le Ch^2t^{-1}\| v
\|,\  \ifff v_h =P_h v. 
\end{equation*}
\end{lemma}
\begin{proof} Using Lemma \ref{lem:de-2-est}  for $\dehat$,  it remains to bound 
$\tde(t)-\wtE_h(t)\Ah Q_h v_h$, which as for
 Lemma \ref{lumped-L2-norm-nonsmooth}, is done as in
\cite[Theorem 4.1]{clt11}.
\end{proof}

The following is now our nonsmooth data error estimate. Its proof
is an obvious modification of that of Theorem  \ref{lumped-nonsmooth-opt},  
using Lemmas \ref{bEh-estimates-2lemma},  \ref{q_a-stability-a} and
\ref{lumped-L2-norm-nonsmooth-1}. 
\begin{theorem}\label{error-gen}
 Let $u$ and $\tu_h$ be the solutions of \eqref{1.eq1-general} and 
\eqref{A_h-problem}, and let $Q_h$
be defined by \eqref{new-Qh}. Then, if 
 \eqref{higher_error} holds, we have
\begin{equation*}
\| \tu_h(t) -  u(t)\| \le Ch^2t^{-1}\| v \|,\ \ifff v_h =P_h v, \for
t>0.
\end{equation*}
\end{theorem}

Condition \eqref{higher_error} on $Q_h$ is again satisfied for
symmetric meshes:
\begin{theorem}\label{gen-sym}
For $\{\T_h\}$  symmetric, \eqref{higher_error} holds for
 $Q_h$ defined by \eqref{new-Qh}.
\end{theorem}
\begin{proof}
We  follow the  steps in the proof  of Theorem \ref{verify}.
For  given $\chi \in S_h$ we define $\varphi=\varphi_\chi\in \dot H^1$ as the solution 
of the Dirichlet problem $\A\vfy=\chi$  in $\Om$,
 $\vfy=0$ on $\partial\Om$. Since $\Om$ is  convex we have
$\varphi \in \dot H^2$ and $|\varphi|_{2} \le C \|\chi\|$. 
For $\psi \in S_h$, we have
\begin{align*}
\| Q_h \psi \| & =\sup_{\chi \in S_h} \frac{(Q_h \psi, \chi)}{\|\chi\|}
  =\sup_{\chi \in S_h}
\frac{ a( Q_h \psi, \vfy)}{\|\chi\|}
\\
  &\le\sup_{\chi \in S_h}
 \frac{|a( \ww, \vfy - I_h \vfy)|}{\|\chi\|}
  +\sup_{\chi \in S_h}
\frac{|a( \ww, I_h \vfy)|}{\|\chi\|}=I+II.
\notag
\end{align*}

By the obvious error estimate for $I_h$ 
and Lemma \ref{q_a-stability-a}, with $p=0$,
we get
\begin{equation*}
| I | \le C h \,\sup_{\chi\in S_h} \frac{\|\nabla Q_h \psi \|  \, |\vfy|_{2}}{\|\chi \| } 
\le C h^2 \|\psi\|.
\end{equation*}
To estimate $II$, we rewrite  the numerator in the form 
\begin{equation*}
 a( \ww, I_h \vfy) =  
-\tveph( Q_h\psi, I_h \vfy)+ \tah( Q_h\psi, \I_h I_h
\vfy)={ii_1+ii_2.}
\end{equation*}
In order to complete the proof it suffices to show that 
\begin{equation*}\label{a-error-0}
|ii_1+ii_2|\le Ch^2\|\chi\|\,\|\psi\|. 
\end{equation*}
Using Lemmas \ref{epal} and \ref{q_a-stability-a} we obtain
\begin{equation*}\label{a-error-1}
|ii_1 | \le C h^2 \|\nabla Q_h \psi \|\,\|
\nabla I_h \vfy\|
  \le C h^2 \|\nabla Q_h \psi \|\, \|\vfy\|_{H^2} \le C h^2 \| \psi \|\, \|\chi\| .
\end{equation*}
Also, employing \eqref{new-Qh} and \eqref{Mh-def0} we get
 \begin{equation*}\label{a-error-2}
ii_2=\vep_h(\psi,I_h\vfy)=[\psi,M_hI_h\vfy].
\end{equation*}
Since the family $\{\T_h\}$ is symmetric, \eqref{eh-bound} 
shows the required bound for $ii_2$.
\end{proof}

The results of Theorems \ref{verify1} and \ref{th4.3}
for our less restrictive assumptions on the family $\{\T_h\}$
also remain
valid, with the obvious modified proofs.

The above results for the spatially semidiscrete finite volume method 
\eqref{1.fv-general-22} extend in the obvious way to the fully
discrete backward Euler method \eqref{lmvvv} and the Crank--Nicolson
method \eqref{cnvvv}, with $-\wtDe_h$ replaced by $\Ah$, so that
Theorems \ref{lmfully}--\ref{lmfully-3} remain literally valid in the
general case.
\section*{Acknowledgements}
The research of R.~D. Lazarov was supported in parts by US NSF Grant
DMS-1016525 and by award KUS-C1-016-04, made by King Abdullah University
of Science and Technology (KAUST). 
The research of P. Chatzipantelidis was partly supported by the
FP7-REGPOT-2009-1 project ``Archimedes Center for Modeling Analysis
and Computation'', funded by the European Commission.
%

\end{document}